\documentclass{amsproc}
\usepackage{hyperref}
\usepackage{color}
\usepackage[dvipdf]{graphicx}
\usepackage{xypic}

\newcommand\ovl[1]{\overline{#1}}

\newcommand{\eps}{\varepsilon}
\newcommand{\Z}{\mathbb{Z}}
\newcommand{\C}{\mathbb{C}}
\newcommand{\Q}{\mathbb{Q}}

\newtheorem{theorem}{Theorem}[section]
\newtheorem{proposition}[theorem]{Proposition}

\theoremstyle{definition}
\newtheorem{definition}[theorem]{Definition}
\newtheorem{remark}[theorem]{Remark}
\newtheorem{example}[theorem]{Example}

\newtheorem{conjecture}[theorem]{Conjecture}
\newtheorem{conjecture/question}[theorem]{Conjecture/Question}

\newtheorem{remark/definition}[theorem]{Remark/Definition}
\newtheorem{terminology/notation}[theorem]{Terminology/Notation}

\definecolor{zielony}{rgb}{0.5, 0.9, 0.1}
\definecolor{czerwony}{rgb}{0.9, 0.2, 0.1}
\definecolor{niebieski}{rgb}{0.3, 0.1, 0.9}

\newcommand{\Oc}{\mathcal{O}_C}
\newcommand{\Pic}{\operatorname{Pic}}
\newcommand{\Nm}{\operatorname{Nm}}
\newcommand{\rk}{\operatorname{rk}}

\newcommand{\rank}{\operatorname{rank}}

\newcommand{\st}{\operatorname{st}}

\newcommand{\ra}{\rightarrow}

\newcommand{\xra}{\xrightarrow}
\newcommand{\Oo}{\mathcal{O}}
\newcommand{\R}{\mathcal{R}}

\newcommand{\M}{\mathcal{M}}

\newcommand{\A}{\mathcal{A}}
\newcommand{\K}{\mathcal{K}}
\newcommand{\SSS}{\mathcal{S}}
\newcommand{\Sp}{\mathcal{S}^+_g}

\newcommand{\Spo}{\overline{\mathcal{S}}^+_g}
\newcommand{\Soo}{\overline{\mathcal{S}}^+_8}
\newcommand{\Mgdr}{\mathcal{M}^r_{g,d}}
\newcommand{\Mgdro}{\overline{\mathcal{M}}^r_{g,d}}

\def\mm{\overline{\mathcal{M}}}
\def\ss{\overline{\mathcal{S}}}

\def\J{\mathcal{J}}
\def\PP{{\bf{P}}}
\def\B{\mathcal{B}}
\def\U{\mathcal{U}}
\def\S{\mathcal{S}}

\begin{document}

\title[Prym varieties and their moduli]
 {Prym varieties and their moduli}

\author[G. Farkas]{Gavril Farkas}

\address{Humboldt-Universit\"at zu Berlin, Institut F\"ur Mathematik,  Unter den Linden 6
\hfill \newline\texttt{}
 \indent 10099 Berlin, Germany} \email{{\tt farkas@math.hu-berlin.de}}
\thanks{These notes are based on lectures delivered in July 2010 in Bedlewo at the \emph{IMPANGA Summer School on Algebraic Geometry} and in January 2011 in Luminy at the annual meeting \emph{G\'eom\'etrie Alg\'ebrique Complexe}. I would like to thank Piotr Pragacz for encouragement and for asking me to write this paper in the first place,  Maria Donten-Bury and Oskar Kedzierski for writing-up a preliminary version of the lectures, as well as Herbert Lange for pointing out to me the historical figure of F. Prym. This work was finalized during a visit at the Isaac Newton Institute in Cambridge. }
\maketitle

\begin{abstract}
This survey discusses the geometry of the moduli space of Prym varieties.  Several applications of Pryms in algebraic geometry are presented.  The paper begins with with a historical discussion of the life and achievements of Friedrich Prym. Topics treated in subsequent sections include singularities and Kodaira dimension of the moduli space, syzygies of Prym-canonical embedding and the geometry of the moduli space $\mathcal{R}_g$ in small genus.
\end{abstract}

\section{Prym, Schottky and 19th century theta functions }

Prym varieties are principally polarized abelian varieties associated to \'etale double covers of curves. They establish a bridge between the geometry of curves and that of abelian varieties and as such, have been studied for over 100 years, initially from an analytic \cite{Wi}, \cite{SJ}, \cite{FR} and later from an algebraic \cite{M} point of view.
Several approaches to the Schottky problem are centered around Prym varieties, see \cite{B2}, \cite{D2} and references therein. In 1909, in an attempt to characterize genus $g$ theta functions coming from Riemann surfaces and thus solve what is nowadays called \emph{the Schottky problem}, F. Schottky and H. Jung, following earlier work of Wirtinger, associated  to certain two-valued \emph{Prym differentials} on a  Riemann surface $C$ new theta constants which then they related to the classical theta constants, establishing what came to be known as the \emph{Scottky-Jung relations}.
The first rigorous proof of the Schottky-Jung relations has been given by H. Farkas \cite{HF}.  The very definition of these differentials forces one to consider the parameter space of unramified double covers of curves of genus $g$. 

The aim of these lectures is to discuss the birational geometry of the moduli space $\R_g$ of Prym varieties of dimension $g-1$.  Prym varieties were named by Mumford  after Friedrich Prym (1841-1915) in the very influential paper \cite{M} in which, not only did Mumford bring to the forefront a largely forgotten part of complex function theory, but he developed an algebraic theory of Pryms, firmly anchored in modern algebraic geometry. In particular, Mumford gave a simple algebraic proof of the Schottky-Jung relations.
\vskip 3pt
 \ \ \ \ To many algebraic geometers Friedrich Prym is a little-known figure, mainly because most of his work concerns potential theory and theta functions rather than algebraic geometry. For this reason, I find it appropriate to begin this article by mentioning a few aspects from the life of this interesting transitional character in the history of the theory of complex functions.

\vskip 4pt
\ \ \ \ \ \ Friedrich Prym was born into one of the oldest business families in Germany, still active today in producing haberdashery articles. He began to study mathematics at the University of Berlin in 1859. After only two semesters, at the advice of Christoffel, he moved for one year to G\"ottingen in order to hear Riemann's lectures on complex function theory. The encounter with  Riemann had a profound effect on Prym and influenced his research for the rest of his life. Having returned to Berlin, in 1863 Prym successfully defended his doctoral dissertation under the supervision of Kummer. The dissertation was praised by Kummer for his didactic qualities,  and  deals with theta functions on a Riemann surface of genus $2$.

\ \ \ \ \ \ After a brief intermezzo in the banking industry, Prym won professorships first in Z\"urich in 1865 and then in 1869 in W\"urzburg, at the time a very small university. In Prym's first year in W\"urzburg, there was not a single student studying mathematics, in 1870 there were only three such students and in the year after that their number increased to four. Prym stayed in W\"urzburg four decades until his retirement in 1909, serving at times as Dean and Rector of the University. In 1872 he turned down a much more prestigious offer of a Chair at the University of Strasbourg, newly created after the Franco-German War of 1870-71. He did use though the offer from Strasbourg in order to improve the conditions for mathematical research in W\"urzburg. In particular, following more advanced German universities like Berlin
or G\"ottingen, he created a \emph{Mathematisches Seminar} with weekly talks. By the time of his retirement, the street in W\"urzburg on which his house stood was already called \emph{Prymstrasse}. Friedrich Prym was a very rich but generous man. According to \cite{Vol}, he once claimed that while one might argue that he was a bad mathematician, nobody could ever claim that he was a bad businessman. In 1911, Prym published at his own expense in 1000 copies his \emph{Magnum Opus} \cite{PR}. The massive 550 page book, written jointly with his collaborator Georg Rost \footnote{Georg Rost(1870-1958) was a student of Prym's and became Professor in W\"urzburg in 1906. He was instrumental in helping Prym write \cite{PR} and after Prym's death was expected to write two subsequent volumes developing a theory of $n$-th order Prym functions. Very little came to fruition of this, partly because Rost's interests turned to astronomy.  Whatever Rost did write however, vanished in flames during the bombing of W\"urzburg in 1945, see \cite{Vol}.} explains the theory of  \emph{Prym functions}, and was distributed by Prym himself to a select set of people. In 1912, shortly before his death, Prym created the \emph{Friedrich Prym Stiftung} for supporting young researchers in mathematics and endowed it with 20000 Marks, see again \cite{Vol}. Unfortunately, the endowment was greatly devalued during the inflation of the 1920's.
\vskip 3pt

\ \ \ \ \   We now briefly discuss the role Prym played in German mathematics of the 19th century. Krazer \cite{Kr} writes that after the death of both Riemann and Roch in 1866, it was left to Prym alone to continue explaining   "Riemann's science" (\emph{"...Prym allein die Aufgabe zufiel, die Riemannsche Lehre weiterzuf\"uhren"}). For instance, the paper \cite{P1} published during Prym's years in Z\"urich, implements Riemann's ideas in the context of hyperelliptic theta functions on curves of any genus, thus generalizing the results from Prym's dissertation in the case $g=2$.  This work grew out of long conversations with Riemann that took place in 1865 in Pisa, where Riemann was unsuccessfully trying to regain his health.
\vskip 3pt

\ \ \ \ During the last decades of the 19th century, Riemann's dissertation of 1851 and his 1857 masterpiece \emph{Theorie der Abelschen Funktionen}, developed in staggering generality, without examples and with cryptic proofs, was still regarded with mistrust as a "book with seven seals" by many people, or even with outright hostility by Weierstrass and his school of complex function theory in Berlin \footnote{Weierstrass' attack centered on Riemann's use of the Dirichlet Minimum Principle for solving boundary value problems. This  was a central point in Riemann's work on the mapping theorem, and in 1870 in front of the Royal Academy of Sciences in Berlin, Weierstrass gave a famous counterexample showing that the Dirichlet functional cannot always be minimized. Weierstrass' criticism was ideological and  damaging, insofar it managed to create the  impression, which was to persist several decades until the concept of Hilbert space emerged, that some of Riemann's methods are not rigorous. We refer to the beautiful book \cite{La} for a thorough discussion. Note that Prym himself wrote a paper \cite{P2} providing an example of a continuous function on the closed disc, harmonic on the interior and which contradicts the Dirichlet Principle.}. In this context, it was important to have down to earth examples, where Riemann's method was put to work. Prym's papers on theta functions played precisely such a role.
\vskip 3pt

\ \ \ \ The following quote is revelatory for understanding Prym's role as an interpreter of Riemann.  Felix Klein \cite{Kl} describes a conversation of his with Prym that took place in 1874, and concerns the question whether Riemann was familiar with the concept of abstract manifold or merely regarded Riemann surfaces as representations of multi-valued complex functions. The discussion seems to have played a significant role towards crystalizing Klein's view of Riemann surfaces as abstract objects:
\emph{"Ich weiss nicht, ob ich je zu einer in sich abgeschlossenen Gesamtauffassung gekommen w\"are, h\"atte mir nicht Herr Prym vor l\"angeren Jahren eine Mitteilung gemacht, die immer wesentlicher f\"ur mich geworden ist, je l\"anger ich \"uber den Gegenstand nachgedacht habe. Er  erz\"ahlte mir, dass die Riemannschen Fl\"achen urspr\"unglich durchaus nicht notwendig mehrbl\"attrige Fl\"achen \"uber der Ebene sind, dass man vielmehr auf beliebig gegebenen krummen Fl\"achen ganz ebenso komplexe Funktionen des Ortes studieren kann, wie auf den Fl\"achen \"uber der
Ebene"} ("I do not know if I could have come to a self-contained conception [about Riemann surfaces], were it not for a discussion some years ago with Mr. Prym, which the more I thought about the subject, the more important it became to me. He told me that Riemann surfaces are not necessarily multi-sheeted covers of the plane, and one can just as well study complex functions on arbitrary curved surfaces as on surfaces over the plane")\footnote{It is amusing to note that after this quote appeared in 1882, Prym denied having any recollection of this conversation with Klein.}.

\vskip 4pt
\ \ \ \ Prym varieties (or rather,  theta functions corresponding to Prym varieties) were studied  for the first time in Wirtinger's monograph \cite{Wi}. Among other things, Wirtinger observes that the theta functions of the Jacobian of an unramified double covering  split into the theta functions of the Jacobian of the base curve, and \emph{new theta functions} that depend on more moduli than the theta functions of the base curve. The first forceful important application of the Prym theta functions comes in 1909 in the important paper \cite{SJ} of Schottky and his student Jung.

\vskip 4pt
\ \ \ \
Friedrich Schottky (1851-1935) received his doctorate in Berlin in 1875 under Weierstrass and Kummer. Compared to Prym, Schottky is clearly a more important and deeper mathematician. Apart from the formulating the \emph{Schottky Problem} and his results on theta functions, he is also remembered today for his contributions to Fuchsian groups, for the \emph{Schottky Groups}, as well as for a generalization of \emph{Picard's Big Theorem} on analytic functions with an essential singularity. To illustrate Schottky's character, we quote from two remarkable letters that Weierstrass wrote. To Sofja Kowalewskaja he writes \cite{Bo}:
"... [Schottky] is of a clumsy appearance, unprepossessing, a dreamer, but if I am not completely wrong, he possesses an important mathematical talent". The following is from a letter to Hermann Schwarz \cite{Bi}: "... [Schottky] is unsuited for practical life.  Last Christams he was arrested for failing to register for military service. After six weeks however he was discharged as being of no use whatsoever to the army. [While the army was looking for him] he was staying in some corner of the city, pondering about linear differential equations whose coefficients appear also in my theory of abelian integrals. So you see the true mathematical genius of times past, with other inclinations" (\emph{"... das richtige mathematische Genie vergangener Zeit mit anderen Neigungen"}).
\vskip 3pt

\ \ \ \ Schottky was Professor in Z\"urich and Marburg, before returning to the University of Berlin in 1902 as the successor of Lazarus Fuchs \footnote{Fuchs' chair was offered initially to Hilbert, but he declined preferring to remain in G\"ottingen after  the university, in an effort to retain him, created a new Chair for his friend Hermann Minkowski. It was in this way that Schottky, as second on the list, was controversially hired at the insistence of Frobenius, and despite the protests of the Minister, who (correctly) thought that Schottky's teaching was totally inadequate (\emph{``durchaus unbrauchbar''}) and would have preferred that the position be offered to Felix Klein instead \cite{Bo}.   Schottky could never be asked to teach beginner's courses, not even in the dramatic years of World War I.}. Schottky remained at the University of Berlin until his retirement in 1922. Due to his personality he could neither attract students nor play a leading role in the German mathematical life and his appointment
can be regarded as a failure that accentuated Berlin's mathematical decline in comparison with G\"ottingen.

\vskip 3pt
\ \ \ \ The paper \cite{SJ} deals with the characterization of theta-constants $\vartheta\Bigl[\begin{array}{c}
\epsilon\\
\delta\\
\end{array}\Bigr](\tau, 0)$ of period matrices $\tau\in \mathfrak{H}_g$ in the Siegel upper-half space that correspond to Jacobians of algebraic curves of genus $g$. Schottky and Jung start with a characteristic of genus $g-1$, that is a pair $\epsilon, \delta\in \{0, 1\}^{g-1}$ and note that
if one completes this characteristic to one of genus $g$ by adding one column in two possible ways, the product of the two theta constants satisfies the following proportionality relation:
$$\vartheta^2\Bigl[\begin{array}{c}
\epsilon\\
\delta\\
\end{array}\Bigr](\Pi, 0)\sim\vartheta\Bigl[\begin{array}{ccc}
\epsilon& 0\\
\delta& 0\\
\end{array}\Bigr](\tau, 0)\cdot \vartheta\Bigl[\begin{array}{ccc}
\epsilon& 0\\
\delta& 1\\
\end{array}\Bigr](\tau, 0)$$
Here $\tau\in \mathfrak{H}_g$ is the period matrix of a Jacobian of a genus $g$ curve but the novelty is that $\Pi\in \mathfrak{H}_{g-1}$ no longer corresponds to any Jacobian but rather to a Prym variety, constructed from an unramified double covering of the curve whose period matrix is $\tau$. This allows one to obtain theta relations for Jacobians in genus $g$ starting with any theta relation in genus $g-1$ (for instance \emph{Riemann's theta formula}). Schottky himself carried out this approach for $g=4$.

\section{The moduli space of Prym varieties}

The main object of these lectures is the moduli space of unramified double covers of genus $g$ curves, that is, the parameter space
$$\R_g =\Bigl\{[C,\eta] \::\:\hbox{$C$ \ is a smooth curve of genus $g$}, \ \eta \in \Pic^0(C)-\{\Oo_C\}, \ \eta^{\otimes 2} = \Oc\Bigr\}.$$
In its modern guise, that is, as a coarse moduli space representing a stack,  this space appears for the first time in Beauville's influential paper \cite{B3}. The choice of the name corresponds to the French word \emph{rev\^etement}.

\ \ \ \ We begin by recalling basic facts about the algebraic theory of Prym varieties. As a general reference we recommend \cite{ACGH} Appendix C, \cite{BL} Chapter 12 and especially \cite{M}.  We fix an integer $g\geq 1$ and denote by
$$\mathfrak{H}_g:=\{\tau\in M_{g, g}(\mathbb C): \tau=^t\tau, \ \mbox{Im }\tau>0\}$$ the Siegel upper half-space
of period matrices for abelian varieties of dimension $g$, hence $\A_g:=\mathfrak{H}_g/Sp_{2g}(\mathbb Z)$.
The \emph{Riemann theta function with characteristics} $\Bigl[\begin{array}{c}
\epsilon\\
\delta\\
\end{array}\Bigr]$ is defined as the holomorphic function $\vartheta:\mathfrak{H}_g\times \mathbb C^g\rightarrow \mathbb C$, where
$$\vartheta\Bigl[\begin{array}{c}
\epsilon\\
\delta\\
\end{array}\Bigr](\tau, z):=\sum_{m\in \mathbb Z^g} \mathrm{exp} \Bigl(\pi i \ ^t(m+\frac{\epsilon}{2})\tau (m+\frac{\epsilon}{2})+2\pi i \ ^t(m+\frac{\epsilon}{2})(z+\frac{\delta}{2})\Bigr),$$
where $\epsilon=(\epsilon_1, \ldots, \epsilon_g), \delta=(\delta_1, \ldots, \delta_g)\in \{0, 1\}^g$.  For any period matrix $\tau \in \mathfrak{H}_g$, the pair
$$\Bigl(A_{\tau}:=\frac{\mathbb C^g}{\mathbb Z^g+\tau\cdot \mathbb Z^g}, \ \ \Theta_{\tau}:=\Bigl\{\vartheta \Bigl[\begin{array}{c}
0\\
0\\
\end{array}\Bigr](\tau, z)=0\Bigr\}\Bigr)$$
defines a principally polarized abelian variety, that is, $[A_{\tau}, \Theta_{\tau}]\in \A_g$.
\vskip 3pt

\ \ \ \ To a smooth curve $C$ of genus $g$, via the Abel-Jacobi isomorphism
\begin{equation}\label{aj}
\mbox{Pic}^0(C)=\frac{H^0(C, K_C)^{\vee}}{H_1(C, \mathbb Z)},
\end{equation}
one associates a period matrix as follows. Let $(\alpha_1, \ldots, \alpha_g, \beta_1, \ldots, \beta_g)$ a symplectic basis of $H_1(C, \mathbb Z)$, and denote by $(\omega_1, \ldots, \omega_g)$ the basis of $H^0(C, K_C)$ characterized by $\int_{\alpha_i}\omega_j=\delta_{ij}$. Then
$\tau:=\Bigl(\int_{\beta_i} \omega_j\Bigr)_{i, j=1}^g\in \mathfrak H_g$ is a period matrix associated to $C$ and $\vartheta\Bigl[\begin{array}{c}
0\\
0\\
\end{array}\Bigr](\tau, 0)$ is the \emph{theta constant} associated to $C$. The theta function $\vartheta\Bigl[\begin{array}{c}
0\\
0\\
\end{array}\Bigr](\tau, z)$ is the (up to scalar multiplication) unique section of the bundle $\Oo_{A_{\tau}}(\Theta_{\tau})$. The theta functions with characteristic $\vartheta\Bigl[\begin{array}{c}
\epsilon\\
\delta\\
\end{array}\Bigr](\tau, z)$ are the unique sections of the $2^{2g}$ symmetric line bundles on $A_{\tau}$ algebraically equivalent to $\Oo_{A_{\tau}}(\Theta_{\tau})$. A very clear modern discussion of basics of theta functions can be found in \cite{BL} Chapter 3. For the case of Jacobians we also recommend \cite{Fay}.
\vskip 5pt

\ \ \ \ Using the Abel-Jacobi isomorphism (\ref{aj}), one can identify torsion points of order $2$ in the Jacobian variety of $C$ with \emph{half-periods} $0\neq \eta \in H_1(C,\Z_2)$. Given a half-period, by taking its orthogonal complement with respect to the intersection product, one obtains a subgroup of index $2$ in  $H_1(C, \mathbb Z)$ which determines an unramified double cover of $C$.  Algebraically, given a line bundle $\eta\in \mathrm{Pic}^0(C)-\{\Oo_C\}$ together with a sheaf isomorphism $\phi:\eta^{\otimes 2}\stackrel{\cong}\rightarrow \Oo_C$, one associates an unramified double
cover $f:\widetilde{C} \ra C$ such that $$\widetilde{C}:=\mbox{Spec}(\Oo_C\oplus \eta).$$
The multiplication in the $\Oo_C$-algebra $\Oo_C\oplus \eta$ is defined via the isomorphism of sheaves $\phi$, that is,
 $$(a+s)\cdot (b+t):=ab+\phi(s\cdot t)+a\cdot t+b\cdot s,$$
 for $a, b\in \Oo_C$ and $s, t\in \eta$.  Note that $f_{*} (\Oo_{\widetilde{C}})=\Oc \oplus\eta$, in particular
 \begin{equation}\label{sj}
 H^0(\widetilde{C}, f^*L)=H^0(C, L)\oplus H^0(C, L\otimes \eta),
 \end{equation}
 for any line bundle $L\in \mathrm{Pic}(C)$.
\vskip 3pt

\ \ \ \ The \'etale double cover $f:\widetilde{C} \ra C$ induces a \emph{norm map} $$\mathrm{Nm}_{f}: \Pic^{2g-2}(\widetilde{C}) \ra \Pic^{2g-2}(C), \ \ \Nm_f\bigl(\Oo_{\widetilde{C}}(D)\bigr):=\Oo_C(f(D)).$$
It is proved in \cite{M} Section 3, that the inverse image $\Nm_f^{-1}(K_C)$ consists of the disjoint union of two copies $$\Nm_f^{-1}(K_C)^{\hbox{\tiny{even}}}\amalg \Nm_f^{-1}(K_C)^{\hbox{\tiny{odd}}}$$ of the same abelian variety, depending on the parity of the number of sections of line bundles on $\widetilde{C}$. We define the Prym variety of the pair $[C, \eta]$ as follows:
\vskip 3pt

\begin{center}
\fbox{$
\mbox{Pr}(C, \eta) := \mbox{Nm}_f^{-1}(K_C)^{\hbox{\tiny{even}}} = \Bigl\{L \in \mbox{Nm}_f^{-1}(K_C): h^0(C, L) \equiv 0\ \mbox{ mod } 2\Bigr\}.
$}
\end{center}
\vskip 3pt
This is a $(g-1)$-dimensional abelian variety carrying a principal polarization. Precisely, if $\Theta_{\widetilde{C}}=W_{2g-2}(\widetilde{C})\subset \mbox{Pic}^{2g-2}(\widetilde{C})$ is the Riemann theta divisor of $\widetilde{C}$, then
\begin{equation}\label{polarization}
\Theta_{\widetilde{C}}\cdot \mathrm{Pr}(C, \eta)=2\ \Xi_C,
\end{equation}
where $\Xi_C$ is a principal  polarization which can be expressed set-theoretically as
$$\Xi_C = \bigl\{L \in \Nm_f^{-1}(K_C)^{\hbox{\tiny{even}}}: h^0(\widetilde{C}, L) \geq 2\bigr\}.$$

\begin{example} We explain how to construct the period matrix of the Prym variety $\mathrm{Pr}(C, \eta)$. Let $(a_0, \ldots, a_{2g-2}, b_0, \ldots, b_{2g-2})$ be a symplectic basis of $H_1(\widetilde{C}, \mathbb Z)$ compatible with the involution $\iota:\widetilde{C}\rightarrow \widetilde{C}$ exchanging the sheets of $f$, that is,
$$\iota_*(a_0)=a_0,\ \iota_*(b_0)=b_0, \ \iota_*(a_i)=a_{i+g-1} \mbox{ and } \iota_*(b_i)=b_{i+g-1} \mbox{ for } i=1, \ldots, g-1.$$
If $(\omega_0, \ldots, \omega_{2g-2})$ is the basis of $H^0(\widetilde{C}, K_{\widetilde{C}})$ dual to the cycles $\{a_i\}_{i=0}^{2g-2}$, then the forms $u_i:=\omega_i-\omega_{i+g-1}$ are $\iota$ anti-invariant and the period matrix of $\mathrm{Pr}(C, \eta)$ is
$$\Pi:=\Bigl(\int_{b_i} u_i\Bigr)_{i, j=1}^{g-1}\in \mathfrak{H}_{g-1}.$$
\end{example}
\vskip 3pt

\begin{example} Having fixed an \'etale covering $f:\widetilde{C}\rightarrow C$ as above, we denote by
$$\Theta_{C, \eta}:=\{L\in \mbox{Pic}^{g-1}(C): h^0(C, L\otimes \eta)\geq 1\}$$ the translate of the Riemann theta divisor and by $f^*:\mbox{Pic}^{g-1}(C)\rightarrow \mbox{Pic}^{2g-2}(\widetilde{C})$ the pull-back map.
The following algebraic form of the \emph{Schottky-Jung relation} holds:
\begin{equation}\label{sj2}
(f^*)^{-1}(\Theta_{\widetilde{C}})=\Theta_C+\Theta_{C, \eta},
\end{equation}
where $\Theta_C=W_{g-1}(C)$. Indeed, this is an immediate consequence of (\ref{sj}), for if $L\in \mbox{Pic}^{g-1}(C)$ satisfies $h^0(\widetilde{C}, f^*L)\geq 1$, then $h^0(C, L)\geq 1$ or $h^0(C, L\otimes \eta)\geq 1$.
\end{example}

\ \ \ \ Putting together formulas (\ref{polarization}) and (\ref{sj2}), one  concludes that there exists a proportionality relation between the theta constants of the Jacobian  $\mbox{Pic}^0(C)$ having period matrix $\tau_g\in \mathfrak{H}_g$,  and those of $\mathrm{Pr}(C, \eta)$ with corresponding period matrix $\Pi_{g-1}\in \mathfrak{H}_{g-1}$. This is the Schottky-Jung relation \cite{SJ}:

\begin{center}
\fbox{$\lambda\cdot \vartheta^2\Bigl[\begin{array}{c}
\epsilon\\
\delta\\
\end{array}\Bigr](\Pi_{g-1}, 0)=\vartheta\Bigl[\begin{array}{ccc}
\epsilon& 0\\
\delta& 0\\
\end{array}\Bigr](\tau_g, 0)\cdot \vartheta\Bigl[\begin{array}{ccc}
\epsilon& 0\\
\delta& 1\\
\end{array}\Bigr](\tau_g, 0)$}
\end{center}
The constant $\lambda\in \mathbb C^*$ is independent of the characteristics $\epsilon, \delta\in \{0, 1\}^{g-1}$.

\vskip 4pt

\ \ \ \ The moduli space  $\R_g$ is thus established as a highly interesting correspondence between the moduli space of curves and the moduli space of principally polarized abelian varieties:
$$\xymatrix{
  & \R_g \ar[dl]_{\pi} \ar[dr]^{\Pr_g} & \\
   \M_g & & \A_{g-1}       \\
                 }$$
Here $\pi$ is forgetful map whereas $\Pr_g$ is the Prym map
$$\R_g\ni [C, \eta]\mapsto \bigl[\mathrm{Pr}(C, \eta), \ \Xi_C\bigr]\in \A_{g-1}.$$
 We denote by $\mathcal{P}_{g-1}$ the closure in $\A_{g-1}$ of the image $\mathrm{Pr}_g(\R_g)$. It is proved in \cite{FS} that the Prym map is generically injective when $g\geq 7$. Unlike the case of Jacobians, $\mathrm{Pr}_g$ is never injective \cite{D2} and the study of the non-injectivity locus of the Prym map is a notorious open problem. Without going into details, we point out \cite{IL} for an important recent result in this direction.

\section{Why $\R_g$?}
Since Mumford \cite{M} "rediscovered" Prym varieties and developed their algebraic theory using modern techniques, there have been a number of important developments in algebraic geometry where Prym varieties and their generalizations play a decisive role. We mention four highlights:

\vskip 3pt

\subsection{The Schottky problem.} The Torelli map
$$t_g:\M_g\rightarrow \A_g, \ \ \ t_g([C]):=[\mathrm{Jac}(C), \Theta_C],$$
assigns to a smooth curve its principally polarized Jacobian variety. It is the content of Torelli's theorem that the map $t_g$
is injective, that is, every smooth curve $C$ can be recovered from the pair $(\mathrm{Jac}(C), \Theta_C)$, see \cite{An} for one of the numerous proofs. To put it informally, the Schottky problem asks for a characterization of the Jacobian locus $\mathcal{J}_g:=\overline{t_g(\M_g)}$ inside $\A_g$. \vskip 3pt

\underline{\emph{Schottky problem (Analytic formulation):}} Characterize the period matrices $\tau\in \mathfrak{H}_{g}$ that correspond to Jacobians. Find equations of the theta constants $\vartheta\bigl[\begin{array}{c}
\epsilon\\
\delta\\
\end{array}\bigr](\tau, 0)$ of Jacobian varieties of genus $g$.
\vskip 4pt

\ \ \ \ Van Geemen \cite{vG} has shown the Jacobian locus $\mathcal{J}_g$ is a component of the locus $\mathcal{SJ}_g\subset \A_g$ consisting of period matrices $\tau\in \mathfrak{H}_g$ for which the Schottky-Jung relations are satisfied for \emph{all} characteristics $\epsilon, \delta\in \{0, 1\}^g$. In the case $g=4$ there is a single Schottky-Jung relation, a polynomial of degree $16$ in the theta constants, and which cuts out precisely the hypersurface $\mathcal{J}_4\subset \A_4$. This is the formula given by Schottky \cite{Sch} in 1888 and one concludes that the following equality holds:
$$\mathcal{J}_4=\mathcal{SJ}_4.$$
Other analytic characterizations of $\mathcal{J}_g$ (KP equation, $\Gamma_{00}$ conjecture of van Geemen-van der Geer) have been recently surveyed by Grushevsky \cite{G2}.
\vskip 4pt

\underline{\emph{Schottky problem (Geometric formulation):}} Find geometric properties of principally polarized abelian varieties that distinguish or single out Jacobians.

\vskip 3pt
\ \ \ \ The most notorious geometric characterization is in terms of singularities of theta divisors. Andreotti and Mayer \cite{AM} starting from the observation that that the theta divisor of the Jacobian of a curve of genus $g$ is singular in dimension at least $g-4$, considered the stratification of $\A_g$ with strata
$$N_{g, k}:=\{[A, \Theta]\in \A_g: \mathrm{dim}(\Theta)\geq k\}$$
and showed that $\mathcal{J}_g$ is an irreducible component of $N_{g, g-4}$. This is what is called a \emph{weak geometric characterization of Jacobians}. Unfortunately, $N_{g, g-4}$ contains other components apart $\mathcal{J}_g$, hence the adjective "weak". The same program at the level of Prym varieties has been carried out by Debarre \cite{De}. Using the methods of Andreotti-Mayer, he showed that
$\mbox{dim}(\Xi_C)\geq g-6$
for any $[C, \eta]\in \R_g$ and for $g\geq 7$, the Prym locus $\mathcal{P}_g$ is a component of $N_{g, g-6}$.

\ \ \ \ To give another significant recent example, Krichever \cite{Kr} found a solution to \emph{Welters' Conjecture} stating that an abelian variety $[A, \Theta]\in \A_g$ is a Jacobian if and only if the Kummer image $\mathrm{Km}: A \stackrel{|2\Theta|}\longrightarrow \PP^{2^g-1}$ admits a trisecant line. Using similar methods, Grushevsky and Krichever \cite{GK} found a characterization of the Prym locus $\mathcal{P}_g$ in terms of quadrisecant planes in the Kummer embedding.
\vskip 5pt

\ \ \ \ By a dimension count, note that $\dim(\M_g)=\dim(\R_g) =3g-3>\mbox{dim}(\mathcal{J}_{g-1})$. One expects to find more Pryms than Jacobians in a given genus, and indeed, it is known that $\mathcal{J}_{g-1} \subset \mathcal{P}_{g-1}$, that is, Jacobians of dimension $g-1$ appear as limits of Prym varieties. We refer to \cite{Wi} for the original analytic proof, or to \cite{B2} for a modern algebraic proof.  Therefore one has at his disposal a larger subvariety of $\A_{g-1}$ than $\mathcal{J}_{g-1}$ which is amenable to geometric study via the rich and explicit theory of curves and their correspondences. This approach is particularly effective for $g \leq 6$,  when $\mathrm{Pr}_g:\R_g \ra \A_{g-1}$ is dominant, hence the study of $\R_g$ can be directly used to derive information about $\A_{g-1}$. It is one of the main themes of these lectures to describe the geometry of $\R_g$ when $g\leq 8$.
\vskip 3pt
\subsection{Rationality questions for $3$-folds.} Prym varieties have been used successfully to detect non-rational Fano $3$-folds. If $X$ is a smooth Fano $3$-fold (in particular $H^{3, 0}(X)=0$), its \emph{intermediate Jacobian} is defined as the complex torus $$\mathcal{J}(X):=H^{2, 1}(X)^{\vee}/H^3(X, \mathbb Z),$$
with the polarization coming from the intersection product on $H^3(X, \mathbb Z)$. Since $H^{3, 0}(X)=0$, one obtains in this way a principally polarized abelian variety. Assume now that $f:X\rightarrow \PP^2$ is a conic bundle and consider the discriminant curve
$$C:=\bigl\{t\in \PP^2: f^{-1}(t)=l_1+l_2, \mbox{ where } l_i\subset X \mbox{ are lines}\bigr\}.$$ Thus $C$ parametrizes pairs of lines, and assuming that $l_1\neq l_2$ for every $t\in C$, we can consider the \'etale double cover $C'\rightarrow C$ from the parameter space of lines themselves to the space classifying pairs of lines. It is then known \cite{B3} that $(\mathcal{J}(X), \Theta_{\J})\cong (\mathrm{Pr}(\widetilde{C}/C), \Xi)$, that is, the intermediate Jacobian of $X$ is a Prym variety. Furthermore, $X$ is rational if and only if $(\mathcal{J}(X), \Theta_{\mathcal{J}})$ is a Jacobian. Using the explicit form of the theta divisor of a Prym variety, in some cases one can rule out the possibility that $\mathcal{J}(X)$ is isomorphic to a Jacobian and conclude that $X$ cannot be rational. In this spirit, Clemens and Griffiths \cite{CG} proved that any \emph{smooth} cubic threefold $X_3 \subset \PP^4$ is non-rational.  Since $\mathcal{J}(X_3)$ is the Prym variety corresponding to a smooth plane quintic, it follows from the study of $\mathrm{Sing}(\Xi)$ carried out in \cite{M}, that  $X_3$ is not rational. A similar approach, works in a number of other cases, e.g. when $X\subset \PP^6$ is a smooth intersection of three quadrics, see \cite{B3}.
\vskip 3pt

\subsection{The Hitchin system.} This is a topic that has seen an explosion of interest recently since Ng\^o \cite{N} proved the fundamental lemma in the Langlands program using the topology of the Hitchin system. We place ourselves in a restrictive set-up just to present certain ideas.
Let $C$ be a smooth curve of genus $g$ and denote by $\M:=\mathcal{SU}_C(2, \Oo_C)$ the moduli space of semistable rank $2$ vector bundles $E$ on $C$ with $\mbox{det}(E)=\Oo_C$. For a point $[E]\in \mathcal{SU}_C(2, \Oo_C)$ with $E$ stable, we have the following identification $T_{[E]}^{\vee}(\M)=\mathrm{Hom}(E, E\otimes K_C)_0$, where the last symbol refers to the homomorphisms of trace zero. The cotangent bundle $T^{\vee}_{\M}$ can be viewed as the space of \emph{Higgs fields} $(E, \phi)$, where
$\phi:E\rightarrow E\otimes K_C$ is a homomorphism. The \emph{Hitchin map} is defined as
$$H:T^{\vee}_{\M}\longrightarrow H^0(C, K_C^{\otimes 2}), \ \ \mbox{  } H(E, \phi)=\mbox{det}(\phi)\in  H^0(C, K_C^{\otimes 2}).$$
It is proved in \cite{H} that the map $H$ is a completely integrable system, and for a general quadratic differential $q\in H^0(C, K_C^{\otimes 2})$, the fibre $H^{-1}([q])$ equals the Prym variety $\mathrm{Pr}(C_q/C)$, where $C_q$ is the spectral curve whose local equation in the total space of the canonical bundle of $C$ is $y^2=q(x)$.
\vskip 3pt

\subsection{Smooth finite Deligne-Mumford covers of $\mm_g$.} Prym level structures have been used by Looijenga \cite{Lo} to construct Deligne-Mumford
Galois covers of $\mm_g$. These spaces are smooth (as varieties, not only as stacks!), modular and can be used to greatly simplify Mumford's definition of intersection products on $\mm_g$. If $S$ is a compact oriented topological surface of genus $g$, its universal Prym cover is a connected unramified Galois cover $\tilde{S}\rightarrow S$ corresponding to the normal subgroup of $\pi_1(S, x)$ generated by the squares of all elements. The Galois group of the cover is denoted by $G:=H_1(S, \mathbb Z_2)$. A \emph{Prym level $n$-structure} on a smooth curve of genus $g$ is a class of orientation preserving homeomorphisms $f:S\rightarrow C$, where two such homeomeorphisms $f, f'$ are identified, if the homeomorphism $f^{-1}\circ f':S\rightarrow S$ has the property that its lift, viewed as an orientation preserving homeomeorphism of $\tilde{S}$, acts as an element of $G$ on $H_1(\tilde{S}, \mathbb Z_n)$. The moduli space $\M_g{n \choose 2}$ of smooth curves with a Prym level $n$-structure is a Galois cover of $\M_g$. Remarkably, the normalization $\mm_g{n\choose 2}$ of $\mm_g$ in the function field of $\M_g{n\choose 2}$ is a \emph{smooth variety} for even $n\geq 6$.  Therefore $\mm_g$ is the quotient of a smooth variety by a finite group!
\vskip 4pt

\section{Parametrization of $\R_g$ in small genus}

We summarize the current state of knowledge about the birational classification of $\R_g$ for small genus. Firstly, $\R_1=X_0(2)$ is rational. The rationality of $\R_2$ is classical and several modern proofs exist in the literature. We sketch the details of one possible approach following \cite{Do2}.

\begin{theorem}\label{Dolgachev} $\R_2$ is rational.
\end{theorem}
\begin{proof}  Suppose that $C$ is a smooth curve of genus $2$. The $15$ non-trivial points of order $2$ on $C$ are in bijective correspondence to sums $p+q\in C_2$ of distinct Weierstrass points on $C$. Thus if $[C, \eta]\in \R_2$, there exist unique Weierstrass points $p\neq q\in C$ such that $K_C\otimes \eta=\Oo_C(p+q)$. One considers the line bundle $L:=K_C^{\otimes 2}\otimes \eta=K_C(p+q)\in \mbox{Pic}^4(C)$. By applying Riemann-Roch, the image of the map
$\phi_L:C\rightarrow \Gamma \subset \PP^2$ is  a plane quartic curve $\Gamma$ with a singular point $u:=\phi_L(p)=\phi_L(q)$.
Furthermore, both vanishing sequences of the linear series $|L|$ at the points corresponding to the node $u$ are equal to
$$a^L_C(p)=a^L_C(q)=(0, 1, 3),$$
that is, the tangent lines at $u$ to the two branches of $\Gamma$, intersect the nodal curve $\Gamma$ with multiplicity $4$. Accordingly, setting $u:=[1:0:0]\in \PP^2$, the plane equation for $\Gamma$ in coordinates $[x:y:z]$ can be given as
$$x^2yz+xyzf_1(y, z)+f_4(y, z)=0,$$
where $f_1(y, z)$(respectively $f_4(y, z)$) is a  linear (respectively quartic) form. By a judicious change of coordinates, one may assume $f_1=0$ and then $\R_2$ is birational to a quotient of $\mathbb C^5$ by a $2$-dimensional torus, which is a rational variety.
\end{proof}

\ \ \ \ For the rationality of $\R_3$ we mention again \cite{Do2}. There exists an alternative  approach due to Katsylo \cite{Ka}. The space $\R_4$ is rational \cite{Ca} and we shall soon return to this case. There are two different proofs of the unirationality of  $\R_5$ in  \cite{IGS} and \cite{V2} respectively.

\vskip 4pt

\ \ \ \  The case of $\R_6$ is the most beautiful and richest from the geometrical point of view. Observing that $\mbox{dim}(\R_6)=\mbox{dim}(\A_5)=15$, one expects the  Prym map $$\mathrm{Pr}_6:\R_6\ra \A_5$$ to be generically finite, therefore also dominant.  By degeneration methods, Wirtinger \cite{Wi} showed this indeed to be the case. Much later, Donagi and Smith \cite{DS} proved that its degree is equal to $27$ which suggests a connection to cubic surfaces. We cannot resist quoting from \cite{DS} p.27: \emph{Wake an algebraic geometer in the dead of the night whispering "27". Chances are, he will respond: "lines on a cubic surface"}.  Donagi \cite{D2} subsequently showed that the Galois group of $\R_6$ over $\A_5$, that is, the monodromy groups of $\mathrm{Pr}_6$, is equal to the Weyl group $W(E_6)\subseteq \mathfrak{S}_{27}$. We recall that $W(E_6)$ is the group of symmetries of the set of lines on a cubic surface. Precisely, if $X\subset \PP^3$ is a fixed smooth  cubic surface and $\{l_1, \ldots, l_{27}\}$ is a numbering of its $27$ lines, then $$W(E_6):=\{\sigma\in \mathfrak{S}_{27}: l_{\sigma(i)}\cdot l_{\sigma(j)}=l_i\cdot l_j \ \mbox{ for all }\  i, j=1, \ldots, 27\}.$$
The statement that $\R_6$ (and hence $\A_5$) is unirational admits at least three very different proofs due to Donagi \cite{D1} using intermediate Jacobians of Fano $3$-folds, Mori and Mukai \cite{MM}, and  Verra \cite{V1} who used the fact that a general Prym curve $[C, \eta]\in \R_6$ can be viewed as a section of an Enriques surface.
\vskip 3pt
\ \ \ \ Having reached this point, one might wonder for which values of $g$ is  $\R_g$ unirational. The following result \cite{FL}
(to be explained in some detail in the next chapters), provides an upper bound on the genus $g$ where one may hope to have an explicit  unirational description of the general Prym curve $[C, \eta]\in \R_g$.

\begin{theorem}
 The Deligne-Mumford compactification $\ovl{\R}_g$ of $\R_g$ is a variety of general type for $g\geq 14$. Furthermore, $\kappa(\ovl{\R}_{12}) \geq 0$, in particular $\ovl{\R}_{12}$ cannot be uniruled.
\end{theorem}

Allowing us to speculate a little further, by analogy with the case of the spin moduli space of curves, it seems plausible that $\ovl{\R}_g$ is not of general type for $g\leq 11$. In what follows we shall confirm this expectation for $g=7, 8$. As for the remaining cases $g=9, 10, 11$, hardly anything seems to be known at the moment.

\subsection{Nikulin K3 surfaces}

Our next aim is to find a uniform way of parametrizing $\R_g$ for $g\leq 7$. To that end, we consider the following general situation.
Let $S$ be a smooth K3 surface and $E_1,\ldots, E_N$ a set of disjoint, smooth rational $(-2)$-curves on $S$. One may ask when is the class  $E_1+\cdots+E_N$ divisible by two (even), that is,
$$e^{\otimes 2}=\Oo_S(E_1+\cdots+E_N)$$
for a suitable class $e\in \mbox{Pic}(S)$.
Equivalently, there exists a double cover $$\epsilon:\widetilde{S}\rightarrow S$$ branched precisely along the curves $E_1, \ldots, E_N$. Note that $\epsilon^{-1}(E_i)\subset \widetilde{S}$ are $(-1)$-curves which can be blown-down and the resulting smooth surface has an automorphism permuting the sheets of the double covering. The answer to this question is due to Nikulin \cite{Ni} and only two cases are possible:

(i) $N=16$ and $\widetilde{S}$ is birational to an abelian surface and $S$ itself is a Kummer surface.

(ii) $N=8$ and  $\widetilde{S}$ is also a K3 surface. In this case $(S, e)$ is called a \emph{Nikulin $K3$ surface}.

\ \ \ For a reference to $K3$ surfaces with an even set of rational curves, we recommend \cite{Ni}, \cite{vGS}, whereas for generalities on moduli space of polarized $K3$ surfaces, see \cite{Do3}.  Suppose that $(S, e)$ is a Nikulin
$K3$ surface and $C\subset S$ is a smooth curve with $C^2=2g-2$, such that $C\cdot E_i=0$ for $i=1, \ldots, 8$. Then the restriction $e_C:=e\otimes \Oo_C$ is a point of order $2$ in the Jacobian of $C$. This link between Nikulin $K3$ surfaces and Prym curves prompts us to make the following definition \cite{FV2}:
\begin{definition}
The moduli space of polarized Nikulin surfaces of genus $g$ is defined as the following parameter space:
$$\mathcal{F}_g^{\mathfrak{N}} = \Bigl\{[S, h, e]: h\in \mathrm{Pic}(S) \mbox{ is  nef  }, \ h^2=2g-2, \ \mathrm{Pic}(S) \supseteq \langle E_1,\ldots,E_8,h\rangle,$$
$$
 e=\frac{1}{2}\Oo_S(\sum_{i=1}^8 E_i)\in \mathrm{Pic}(S),\  h\cdot E_i = 0, \ E_i^2=-2, \ E_i\cdot E_j=0 \ \mbox{ for }i\neq j \Bigr\}.$$
\end{definition}
Note that $\mathcal{F}_g^{\mathfrak{N}}$ is an irreducible variety of dimension $11$, see  \cite{Do3}. Nikulin surfaces depend on $11$ moduli because polarized $K3$ surfaces of genus $g$ depend on $19$ moduli, from which one subtracts $8$, corresponding to the number of independent condition being imposed on the lattice $\mathrm{Pic}(S)$. We then consider the  $\PP^g$-bundle over $\mathcal{F}_g^{\mathfrak{N}}$
$$\mathcal{P}_g^{\mathfrak N} := \Bigl\{\bigl([S, h, e], C\bigr): [S, h, e]\in \mathcal{F}_g^{\mathfrak{N}}, \ \ C\subset S, \ C \in |h|\Bigr\}.$$
The restriction line bundle $e_C:=e\otimes C\in \mathrm{Pic}^0(C)_2$ induces an \'etale double cover $\widetilde{C} \ra C.$
As explained above, we have two morphism between moduli spaces
$$\xymatrix{
  & \mathcal{P}_g^{\mathfrak{N}} \ar[dl] \ar[dr]^{\chi_g} & \\
   \mathcal{F}_g^{\mathfrak{N}} & & \R_{g}       \\
                 }$$
where $\chi_g \bigl([S, e, h],C\bigr) := [C, e_C].$
Note that $$\dim(\mathcal{P}_g^{\mathfrak{N}})=11+g \mbox{   and    } \ \dim(\R_g)=3g-3,$$ hence $\dim(\mathcal{P}_g^{\mathfrak{N}}) \geq \dim(\R_g)$ exactly for $g \leq 7$.
It is natural to ask whether in this range $\mathcal{P}_g^{\mathfrak{N}}$ dominates $\R_g$. Since by construction $\mathcal{P}_g^{\mathfrak{N}}$ is a uniruled variety, this would imply (at the very least) the uniruledness of $\R_g$. At this point we would like to recall the following well-known theorem due to Mukai \cite{M1}:
\begin{theorem}\label{mukaik3}
A general curve $[C]\in \M_g$ appears as  a section of a $K3$ surface precisely when $g\leq 11$ and $g\neq 10$.
For $g=10$, the locus
$$\K_{10}:= \{[C]\in \M_{10}: C \hbox{ lies on a K3 surface}\}$$ is a divisor.
\end{theorem}

The fact that the general curve $[C]\in \M_{10}$ does not lie on a $K3$ surface comes as a surprise, and is due to the existence of the rational homogeneous $5$-fold $$X:=G_2/P\subset \PP^{13}$$ such that $K_X=\Oo_X(-3)$. Thus codimension $4$ linear sections of $X$ are canonical curves of genus $10$; if a curve $[C]\in \M_{10}$ lies on a $K3$ surface, then it lies on a $3$-dimensional family of $K3$ surfaces. This affects the parameter count for genus $10$ sections of $K3$ surfaces and one computes that $$\mbox{dim}(\K_{10})=19+g-3=26=\mbox{dim}(\M_{10})-1.$$ The divisor $\K_{10}$ plays an important role in the birational geometry of $\mm_g$. It is an extremal point of the effective cone of divisors of $\mm_{10}$ and it was the first counterexample to the Harris-Morrison Slope Conjecture, see \cite{FP}.
\vskip 4pt

\ \ \ In joint work with A. Verra \cite{FV2}, we have shown that one has  similar results (and much more) for Prym curves, the role of ordinary $K3$ surfaces being played by
Nikulin surfaces. The following result is quoted from \cite{FV2}:

\begin{theorem}\label{nikulin}
We fix an integer $g \leq 7$, $g \neq 6$. A general Prym curve $[C,\eta]\in \R_g$ lies on a Nikulin surface, that is, the rational map $\chi_g:\mathcal{P}_g^{\mathfrak{N}}\rightarrow \R_g$ is dominant.
\end{theorem}

\noindent
\begin{proof} We discuss the proof only in the case $g=7$ and start with a general element $[C, \eta]\in \R_7$. We consider the Prym-canonical
embedding $$\phi_{K_C\otimes \eta}:C\rightarrow \PP^5.$$ Note that $K_C\otimes \eta$ is very ample, for otherwise $\eta\in C_2-C_2$, in particular $C$ is tetragonal, which contradicts the generality assumption on the pair $[C, \eta]$.
It is shown in \cite{FV2} that  $h^0(\PP^5, \mathcal I_{C / \PP^5}(2)) = 3$, that is,  $|\mathcal I_{C /\PP^5}(2)|$ is a net of quadrics and the base locus of this net is a smooth K3 surface $S\subset \PP^5$.

\ \ \ We claim that $S$ is a Nikulin $K3$ surface.  Let $L := \Oo_S(2H -2C)\in \mathrm{Pic}(S)$ and consider the standard exact sequence
$$
0 \longrightarrow L\otimes \Oo_S(-C) \longrightarrow  L \longrightarrow L\otimes \Oo_C \longrightarrow 0.
$$
Note that $L\otimes \Oo_C=\Oo_C(2H-2K_C)=\Oo_C$. Furthermore  $h^1(S, L\otimes \Oo_S(-C)) = h^1(S, 2H-C) = 0$, because $C$ is quadratically normal. Passing to the long exact sequence, it follows $h^0(S, L) = 1$. The numerical characters of $L$ can be computed as follows:
 $L\cdot H = 8$ and $L^2 = -16$. After analyzing all possibilities, it follows that $L$ is equivalent to the sum of 8 disjoint lines.
 Furthermore $\eta=\Oo_C(C-H)$, which proves that $[C, \eta]=\chi_7([S, \Oo_S(C-H)]$, that is, $C$ lies on a Nikulin surface.
\end{proof}

\ \ \ \ Theorem \ref{nikulin} shows that $\R_g$ is uniruled for $g\leq 7$.
As in Mukai's Theorem \ref{mukaik3}, the genus next to maximal, proves to be exceptional. Let us denote by $\mathcal{N}_6:=\mbox{Im}(\chi_6)$ the locus of Prym curves which are sections of Nikulin surfaces.
\begin{theorem}
One has the following identification of effective divisors on $\R_6$:
$$\mathcal{N}_6= \Bigl\{[C,\eta] \in \R_6:\mathrm{Sym}^2 H^0(C, K_C\otimes \eta) \rightarrow H^0(C, K_C^{\otimes 2}) \mbox{ is not an isomorphism}\Bigr\}.$$ This locus equals the ramification divisor of the Prym map $\mathrm{Pr}_6:\R_6\ra \A_5$.
\end{theorem}

The previous method can no longer work when $g\geq 8$ because Nikulin sections form a locus of codimension at least $2$ in $\R_g$.
Instead we shall sketch an approach to handle the case of $\R_8$ . Full details will appear in the paper \cite{FV3}.

\section{Prym Brill-Noether loci and the uniruledness of $\R_8$}

Some preliminaries on Brill-Noether theory for Prym curves and lagrangian degeneracy loci are needed, see  \cite{M} and \cite{W} for a detailed discussion. We fix a Prym curve $[C, \eta]\in \R_g$ and let $f:\widetilde{C} \ra C$ be the induced \'etale covering map. For a fixed integer $r\geq -1$,  the
 \emph{Prym-Brill-Noether} locus is defined as the following determinantal subvariety of $\Nm_f^{-1}(K_C)$:

\begin{center}
\fbox{$V^r(C,\eta) := \bigl\{L\in \mbox{Nm}_f^{-1}(K_{C}): h^0(\widetilde{C}, L)\geq r+1, \hbox{ } h^0(\widetilde{C}, L)\equiv r+1 \ \mbox{mod}\ 2\bigr\}.$}
\end{center}

The expected dimension of $V^r(C, \eta)$ as a determinantal variety equals $g-1 - {r+1\choose 2}$. In any event, the inequality
$$\mbox{codim}\Bigl(V^r(C, \eta), \mathrm{Pr}(C, \eta)\Bigr)\leq {r+1\choose 2}$$
holds. Note that $V^{-1}(C, \eta)=V^0(C, \eta)=\Pr(C, \eta)$ and $V^1(C, \eta)=\Xi_C$ is the theta-divisor of the Prym variety, in its Mumford incarnation.
\vskip 3pt

\ \ \ \ We fix a line bundle $L \in \Nm_f^{-1}(K_C)$. This last condition is equivalent to $\iota^*L = K_{\widetilde{C}}\otimes L^{\vee}$, where $\iota:\widetilde{C}\rightarrow \widetilde{C}$ is the involution exchanging the sheets of the covering $f:\widetilde{C}\rightarrow C$. Let us consider the Petri map
$$\mu_0(L): H^0(\widetilde{C}, L)\otimes H^0(\widetilde{C}, K_{\widetilde{C}}\otimes L^{\vee}) \rightarrow  H^0(C, K_{\widetilde{C}}).$$
Using the decomposition $H^0(\widetilde{C}, K_{\widetilde{C}})=H^0(C, K_C)\oplus H^0(C, K_C\otimes \eta)$, we can split the Petri map into a $\iota$ anti-invariant part
$$\mu_0^-(L): \Lambda^2 H^0(\widetilde{C}, L) \ra H^0(C, K_C\otimes \eta), \ \ \mbox{  } \ s\wedge t\mapsto s\cdot \iota^*(t)-t\cdot \iota^*(s),$$ and a
$\iota$ invariant part respectively
$$\mu_0^+(L): \mathrm{Sym}^2H^0(\widetilde{C}, L) \ra H^0(C, K_C),\ \ \mbox{  }  s\otimes t+t\otimes s\mapsto s\cdot \iota^*(t)+t\cdot \iota^*(s).$$
Welters calls $\mu_0^-(L)$ the \emph{Prym-Petri} map. The name is appropriate because analogously to the classical Petri map, via the standard identification
$$T_L\Bigl(\mathrm{Pr} (C, \eta)\Bigr)=H^0(C, K_C\otimes \eta)^{\vee}$$ coming from Kodaira-Spencer theory, the map
$\mu_0^{-}(L)$ governs the deformation theory of the loci $V^r(C, \eta)$. We mention the following result, see \cite{W} Proposition 1.9:
\begin{proposition}
Let $L\in \mathrm{Nm}_f^{-1}(K_C)$ with $h^0(\widetilde{C}, L)=r+1$. The Zariski tangent space $T_L(V^r(C, \eta))$ can be identified to
$\Bigl(\mathrm{Im }\  \mu_0(L)^-\Bigr)^\bot$. In particular, $V^r(C, \eta)$ is smooth and of the expected dimension $g-1-{r+1\choose 2}$ at the point
$L$ if and only if $\mu_0^-(L)$ is injective.
\end{proposition}
The main result of \cite{W} states that for a general point $[C, \eta]\in \R_g$, the Prym-Petri map $\mu_0^-(L)$ is injective for every $L\in V^r(C, \eta)$.
In particular, $$\mbox{dim } V^r(C, \eta)=g-1-{r+1\choose 2}.$$ The class of  $V^r(C, \eta)$ has been computed by De Concini and Pragacz \cite{DP}. If $\xi=\theta'/2\in H^2(\mathrm{Pr}(C, \eta), \mathbb Z)$ is the class of the principal theta-divisor of $\mathrm{Pr}(C, \eta)$, then
$$\bigl[V^r(C, \eta)\bigr]=\frac{1}{1^r\cdot 3^{r-1}\cdot 5^{r-2}\cdot \ldots \cdot (2r-1)}\xi^{r+1\choose 2}.$$
This formula proves that $V^r(C, \eta)\neq \emptyset$ when $g-1\geq {r+1\choose 2}$.

\vskip 4pt
\ \ \ \ The map $\mu_0^-(L)$, enjoying this deformation-theoretic interpretation, has received a lot of attention. By contrast, its even counterpart $\mu_0^+(L)$ seems to have been completely neglected so far, but this is what we propose to use in order to parametrize $\R_g$ when $g=8$. We define the universal Prym-Brill-Noether locus
$$\R_g^r:=\Bigl\{\bigl([C, \eta], L\bigr): [C, \eta]\in \R_g,\  L\in V^r(C, \eta)\Bigr\}.$$
When $g-1-{r+1\choose 2}\geq 0$, the variety $\R_g^r$ is irreducible, generically smooth of dimension $4g-4-{r+1\choose 2}$ and mapping dominantly onto $\R_g$.

\vskip 4pt

\ \ \ We now fix $g\geq 4$ and turn our attention to the space $\R_g^2$ which has relative dimension $g-4$ over $\R_g$. A general point $\bigl([C, \eta], L\bigr)\in \R_g^2$, corresponds to a general Prym curve $[C, \eta]\in \R_g$ and a base point free line bundle $L$ such that $h^0(\widetilde{C}, L)=3$.
Setting $\PP^2:=\PP\bigl(H^0(L)^{\vee}\bigr)$, we have the following  commutative diagram:
\begin{figure}[h]
$$\xymatrix@R=16pt{\tilde{C} \ar[rr]^-{(L, \iota^{*} L)} \ar[dd]_f && \PP^2\times \PP^2 \ar[dd]_s \ar@{^{}->}[rd] \\
&& & \PP^8=\PP\bigl(H^0 (L)^{\vee}\otimes H^0(L)^{\vee}\bigr) \ar@{-->}@<-1ex>[dl] \\
C\ar[rr]^-{|\mu_0^+(L)|} && *!<-22pt,0pt>{\PP^5=\PP(\mathrm{Sym}^2 H^0(L)^{\vee})}}
$$
\end{figure}

In the above diagram $s$ is a $2:1$ quasi-\'etale morphism and $\mathrm{Im}(s) = D\subset \PP^5$ is the determinantal cubic hypersurface. The branch locus of $s$ is the Veronese surface $\mathrm{Sing}(D)= V_4 \subseteq \PP^5$. It is well-known that $D$ can be identified with the secant variety of $V_4$. For a general $[C, \eta]\in \R_g$ as above, one can show that
$\mu_0^+(L)$ is injective, that is,  $$W:=\mathrm{Sym}^2 H^0(\widetilde{C}, L)\subset H^0(C, K_C)$$ is a $6$-dimensional space of canonical forms on $C$. The map $s$ is given by
$$\PP^2\times \PP^2\ni \Bigl([a], \ [b]\Bigr)\mapsto [a\otimes b+b\otimes a]\in \PP^5.$$ Equivalently, if $\PP^2$ is viewed as the space of lines in $\PP(H^0(L))$, then $s$ maps a pair of lines $([a], [b])$ to the degenerate conic $[a]+[b]\in \PP^5$. Moreover, $D$ is viewed as the space of degenerate conins in $\PP(H^0(L))$.

\ \ \ \
The commutativity
of the diagram implies that the map $\widetilde{C}\stackrel{|W|}\rightarrow \PP^5$ induced by the sections in $W$ has degree $2$ and factors through $C$. The image curve, which is a projection of the canonical model of $C$, lies on the symmetric cubic hypersurface $D$. Before turning to the case $g=8$, we mention the following result:

\begin{theorem}[Verra 2008]
$\R^2_g$ is a unirational variety for $g \leq 7$.
\end{theorem}

This of course gives a new proof of the unirationality of $\R_g$ when $g\leq 7$. We turn our attention to the case of $\R_8$ and ask when is the image $$C\stackrel{|W|}\longrightarrow \PP^5$$ contained in a $(2,2,3)$ complete intersection, that is, we require that $C$ be contained in two additional quadrics. The idea of showing uniruledness of a moduli space of curves by realizing its general point as a section of a \emph{canonical surface} is not new and has already been used  in \cite{BV} to prove that $\M_{15}$ is rationally connected.  If $S\subset \PP^r$ is a canonical surface and $C\subset S$ is a curve such that $h^1(C, \Oo_C(1))\geq 1$, then $\mbox{dim }|C|\geq 1$, in particular $C$ deforms in moduli, and through a general point of the moduli space there passes a rational curve.
\vskip 3pt

\ \ \ \ To estimate the number of moduli of Prym curves lying on a $(2, 2, 3)$ complete intersection in $\PP^5$, we consider the following morphism between two vector bundles over an open subset of $\R_g^2$:
$$\xymatrix{
  & \mathcal{E}(C, \eta, L)=\mathrm{Sym}^2(W) \ar[rd] \ar[rr]^{\mu} & & H^0(C, K_C^{\otimes 2})=\mathcal{F}(C, \eta, L) \ar[ld] & \\
      &  & \R_g^2 \ar[d]       \\
      & & \R_g
                 }$$
Both $\mathcal{E}$ and $\mathcal{F}$ are vector bundles over $\R_g^2$, with fibres over a point $[C, \eta, L]$ as in the diagram above. The vector bundle morphism $\mu: \mathcal{E}\rightarrow \mathcal{F}$ is given by multiplication of sections. Note that when $g=8$, both $\mathrm{Sym}^2(W)$ and $H^0(C, K_C^{\otimes 2})$  have dimension 21, and we expect the corank $2$ degeneracy locus of $\mu$ to be of codimension $4$ in $\R_g^2$, and hence map dominantly onto $\R_g$.

\vskip 3pt
After some rather substantial work, we can show that a general $[C,\eta] \in \R_8$ lies on a finite number of surfaces $C\subset S\subset \PP^5$, where $S = Q_1\cap Q_2 \cap D \subseteq \PP^5$. Singular points of S are the 16 nodes corresponding to the intersections of $Q_1\cap Q_2$ with the Veronese surface
 $V_4:=\mathrm{Sing}(D)$. Furthermore, $K_S = \Oo_S(1)$ and from the adjunction formula we find that $\Oo_C(C)=\Oo_C$, hence there is an exact sequence
 $$0\longrightarrow \Oo_S\longrightarrow \Oo_S(C) \longrightarrow \Oo_C \longrightarrow 0.$$ One finds that $\dim |\Oo_S(C)| = 1$, that is, $C$ moves in a pencil of curves on $S$. Since the torsion line bundle $\eta$ can be recovered from the projection $s$, we obtain in fact a pencil in $\R_8$, passing through a general point. One has the following result, full details of which will appear in the forthcoming \cite{FV3}:
 \begin{theorem}
 The moduli space $\R_8$ is uniruled.
\end{theorem}

\section{The Kodaira dimension of moduli of Prym varieties}

The aim of this lecture is to show that $\R_g$ is a variety of general type for $g\geq 14$ and to convey, in an informal setting,  some of the ideas contained \cite{FL}, \cite{HM}, \cite{EH}. First we discuss a general program of showing that a moduli space is of general type. This strategy has been used by Harris and Mumford in the case of $\mm_g$, by Gritsenko, Hulek and Sankaran \cite{GHS} in the case of the moduli space of polarized $K3$ surfaces and by the author \cite{F3} in the case of the space $\overline{\mathcal{S}}_g^+$ classifying even theta-characteristics.
\vskip 3pt

\ \ Any attempt to compute the Kodaira dimension of the moduli spaces of Prym varieties must begin with the construction of a suitable compactification of $\R_g$. This compactification should satisfy a number of minimal requirements:
\vskip 3pt

$\bullet$ The covering $\R_g\rightarrow \M_g$ should extend to a \emph{finite branched} covering $\ovl{\R}_g\rightarrow \ovl{\M}_g$.

$\bullet$ Points in $\ovl{\R}_g$ ought to have modular meaning. Ideally, $\ovl{\R}_g$ should be the coarse moduli space associated to a Deligne-Mumford stack
of \emph{stable Prym curves} of genus $g$, that is, points in the boundary should correspond to mildly singular curves with some level structure. If this requirement is fulfilled, one can carry out intersection theory on $\ovl{\R}_g$ and the results have enumerative meaning in terms of curves and their associated Prym varieties.

$\bullet$ The singularities of $\ovl{\R}_g$ should be manageable, in particular we would like pluri-canonical forms defined on the locus of smooth points $\ovl{\R}_{g, \mathrm{reg}}$ to extend to a resolution of singularities of $\ovl{\R}_g$. This implies that the Kodaira dimension of $\ovl{\R}_{g}$, defined as the Kodaira dimension of a non-singular model, coincides with the Kodaira-Iitaka dimension of the canonical divisor $K_{\ovl{\R}_g}$, which is computed at the level of $\ovl{\R}_g$. In practice, this last requirement forces $\ovl{\R}_g$ to have finite quotient singularities.
\vskip 4pt

\ \ \ \ In what follows we describe a satisfactory solution to this list of requirements. We fix a genus $g \geq 2$ and a level $l\geq 2$. We consider the
following generalization of the level $l$ modular curve
$$\R_{g,l} = \Bigl\{[C,\eta] \::\: [C] \in \M_g,\  \eta \in \mathrm{Pic}^0(C)-\{\Oo_C\} \text{ is a point of order } l\Bigr\}.$$
Obviously $\R_{g,2}=\R_g$. There is a forgetful map $\R_{g, l} \xra{\pi} \M_g$ of degree $l^{2g} -1$. The moduli space of twisted stable maps
$\ovl{\R}_{g, l}:=\ovl{\M}_g(\mathcal{B} \Z_l)$ viewed as a compactification of of $\R_{g, l}$ can be fitted into the following  commutative diagram, see \cite{ACV}:
$$\xymatrix{
  & \R_{g, l} \ar[d]_{\pi} \ar@{^{(}->}@<-0.5ex>[r] & \ovl{\R}_{g,l} \ar[d]_{\pi}& \\
  & \M_g \ar@{^{(}->}@<-0.5ex>[r] & \ovl{\M}_g & \\
}$$
By analogy with the much studied case of elliptic curves, one may regard $\ovl{\R}_{g, l}$  as a higher genus generalization of the modular curve $X_1(N)$.
For simplicity we shall explain the construction of $\R_{g, l}$ only in the case $l=2$, and refer to \cite{ACV}, \cite{CCC} for details for the case $l\geq 3$.
\begin{definition}
If  $X$ is a semi-stable curve, a component $E\subseteq X$ is called \emph{exceptional} if $E\simeq \PP^1$ and $E\cap \ovl{(X\setminus E)} = 2$,
that is, it meets the other components in exactly two points. The curve $X$ is called \emph{quasi-stable} if any two exceptional components of $X$ are disjoint.
\end{definition}
If $X$ is a quasi-stable curve, the stable model $\mathrm{st}(X)$ of $X$ is obtained by contracting all exceptional components of $X$.
\begin{definition}
The moduli space of stable Prym curves $\ovl{\R}_{g}$ of genus $g$ parametrizes triples $[X, \eta, \beta]$ such that:
\begin{itemize}
\item $X$ is a quasi-stable curve with $p_a(X) = g$.
\item $\eta \in \Pic^0(X)$, that is, $\eta$ is a locally free sheaf on $X$ of \emph{total} degree $0$.
\item $\eta_E \simeq \Oo_E(1)$ for all exceptional components $E\subseteq X$.
\item $\beta:\eta^{\otimes 2} \ra \Oo_X$ is a sheaf homomorphism which is an isomorphism along non-exceptional components;
\end{itemize}
\end{definition}
Next we define a stack/functor of stable Prym curves, whose associated coarse moduli space is precisely $\ovl{\R}_g$:
\begin{definition}
A \emph{family of Prym curves} over a base scheme $S$ consists of a
triple $(\mathcal{X}\stackrel{f}\rightarrow S, \eta, \beta)$, where
$f:\mathcal{X}\rightarrow S$ is a flat family of quasi-stable
curves, $\eta\in \mathrm{Pic}(\mathcal{X})$ is a line bundle and
$\beta:\eta^{\otimes 2}\rightarrow \Oo_{\mathcal{X}}$ is a sheaf
homomorphism, such that for every point $s\in S$ the restriction
$(X_s, \eta_{X_s}, \beta_{X_s}:\eta_{X_s}^{\otimes 2}\rightarrow
\Oo_{X_s})$ is a stable Prym curve of genus $g$.
\end{definition}
\begin{remark}
Note that by replacing in this definition the structure sheaf $\Oo_X$ by the dualizing sheaf $\omega_X$, we obtain the moduli of stable spin curves $\ovl{\SSS}_g$. Different compactifications of
the space $\R_{g,l}$ were studied  for $l\ge 3$ by Caporaso-Casagrande-Cornalba \cite{CCC}, Jarvis \cite{J} and Abramovich-Corti-Vistoli \cite{ACV}.
\end{remark}

There exists a forgetful morphism of stacks $\pi: \ovl{\R}_{g} \ra \ovl{\M}_g$ which at the level of sets is given by  $\pi([X,\eta,\beta]) = [\st(X)]$. Even though the morphism $\R_g\rightarrow \M_g$ is \'etale (at the level of stacks), the compactification $\ovl{\R}_g\rightarrow \ovl{\M}_g$  is ramified along the boundary. This accounts for better positivity properties of the canonical bundle $K_{\ovl{R}_g}$. As $g$ increases, $\ovl{\R}_g$ is expected to become sooner of general type than $\ovl{\M}_g$.
\vskip 3pt

\ \ \ \ As usual we denote by $\Delta_0 \subseteq \ovl{\M}_g$  the closure of the locus of irreducible one-nodal curves and for $1\leq i\leq [g/2]$ we denote by $\Delta_i\subset \ovl{\M}_g$ the boundary divisor whose general point corresponds to the union of two curves of genus $i$ and $g-i$ respectively, meeting transversally at a single point.

\vskip 4pt
\begin{example}\label{exad}
Let us take a general point  $[C_{xy}]\in\Delta_0$, corresponding to a  normalization map $\nu:C\xra{} C_{xy}$, where $C$ is a curve of genus $g-1$ and $x, y\in C$ are distinct points. We aim to describe all points $[X,\eta,\beta] \in \pi^{-1}([C_{xy}])$. Depending on whether $X$ contains an exceptional component or not, one distinguishes two cases:

If $X = C_{xy}$, there is an exact sequence
$$1\ra \Z_2 \ra \Pic^0(C_{xy})_2 \xra{\nu^*} \Pic^0(C)_2 \ra 1.$$ Setting $\eta_C = \nu^*(\eta)\in \mathrm{Pic}^0(C)$, there are two subcases to be distinguished:
\vskip 3pt
(I) If $\eta_C\neq \Oo_C$,  there is a $\Z_2$-ambiguity (coming from the previous sequence) in identifying the fibres $\eta_C(x)$ and $\eta_C(y)$, that is, there exist two possibilities of lifting $\eta_C$ to a line bundle on $X$.  Such an identification, together with the choice of line bundle $\eta_C\in \mathrm{Pic}^0(C)_2$ uniquely determine a line bundle $\eta$ on $C_{xy}$ which is a square root of the trivial bundle. We denote by $\Delta_0^{'} \subset \ovl{\R}_{g,2}$ the divisor consisting of such stable Prym curves together with all their degenerations.
\vskip 3pt

(II) If $\eta_C = \Oc$, then  there is exactly one way of identifying $\eta_C(x)$ and $\eta_C(y)$ such that $\eta\neq\mathcal{O}_X.$ The closure of this locus is a divisor denoted $\Delta_0^{''} \subset \ovl{\R}_{g,2}$. Points in
$\Delta_0^{''}$ are sometimes called \emph{Wirtinger double covers} \cite{Wi}, since they were used in \cite{Wi} to prove that Jacobians of genus $g-1$ are limits of Prym varieties of genus $g$.
\vskip 3pt

(III) If $X = C \cup_{\{x,y\}} E$, where $E=\PP^1$, then $\eta_E = \Oo_E(1)$ and $\eta_C \in \sqrt{\Oc(-x-y)}$. In this case there is no ambiguity in identifying the fibres and  the corresponding locus is the ramification divisor $\Delta_0^{\mathrm{ram}}$ of the map $ \pi:\ovl{\R}_g \ra \ovl{\M}_g$.
\end{example}
\vskip 3pt

 Keeping the notation above, if $\delta_0 = [\Delta_0]\in \mathrm{Pic}(\ovl{\M}_g)$ and $\delta_0':=[\Delta_0'], \delta_0^{''}:=[\Delta_0^{''}], \delta_0^{\mathrm{ram}}:=[\Delta_{0}^{\mathrm{ram}}]\in \mathrm{Pic}(\ovl{\R}_g)$, we obtain  the following relation:
$$\pi^*(\delta_0) = \delta_0^{'} + \delta_0^{''} + 2\delta_0^{\mathrm{ram}}.$$

\ \ \ All three cases described in Example \ref{exad} correspond to certain types of admissible double covers in the sense of \cite{B2}. These coverings are represented schematically as follows:

\begin{figure}[!h]
\begin{center}
\hspace{2cm}
\includegraphics[width=8cm, height=10cm]{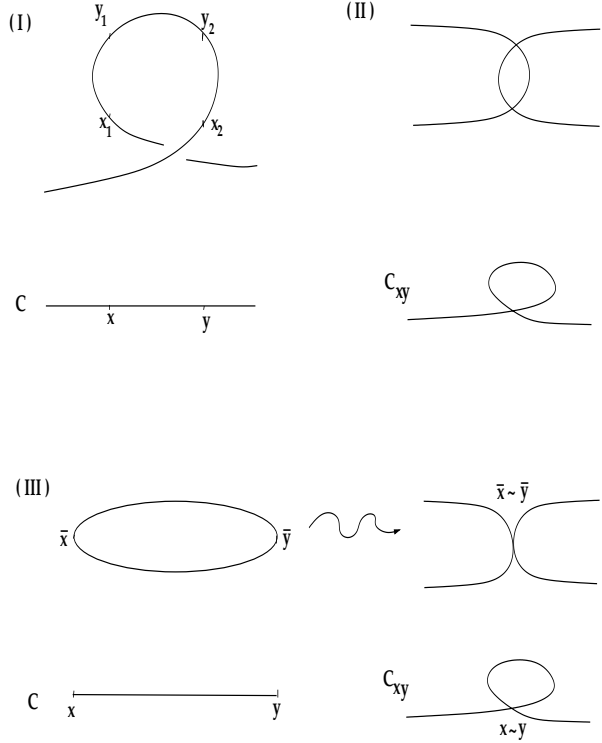}
\caption{admissible double covers}
\end{center}
\end{figure}

\begin{example} (\emph{Curves of compact type}) Let us consider a union of two smooth curves $C$ and
$D$ of genus $i$ and $g-i$ respectively meeting transversally at a point. We describe the fibre $\pi^{-1}([C\cup D])$. If $[X, \eta, \beta]\in \ovl{\R}_{g}$ is a stable Prym curve having as underlying model a curve $X$ with $\mathrm{st}(X)=C\cup D$, first we observe that
that $X=C\cup D$ (that is, $X$ has no exceptional components).
The line bundle $\eta$ on $X$  is determined by the choice of two line
bundles $\eta_C\in \mathrm{Pic}^0(C)$ and $\eta_D\in
\mathrm{Pic}^0(D)$ satisfying $\eta_C^{\otimes 2}=\Oo_C$ and
$\eta_D^{\otimes 2}=\Oo_D$ respectively. This shows that for $1\leq
i\leq [g/2]$ the pull-back under $\pi$ of the boundary divisor
$\Delta_i\subset \ovl{\M}_g$ splits into three irreducible components
$$\pi^*(\Delta_i)=\Delta_i+\Delta_{g-i}+\Delta_{i:g-i},$$ where the
generic point of $\Delta_i\subset \ovl{\R}_g$ is of the form $[C\cup D,
\eta_C\neq \Oo_C, \eta_D=\Oo_D]$, the generic point of
$\Delta_{g-i}$ is of the form $[C\cup D, \eta_C=\Oo_C, \eta_D\neq
\Oo_D])$, and finally $\Delta_{i: g-i}$ is the closure of the locus
of points $[C\cup D, \eta_C\neq \Oo_C, \eta_D\neq \Oo_D]$.
\end{example}

\ \ \ \ The canonical class $K_{\ovl{\R}_g}$ can be computed using the Grothendieck-Riemann-Roch formula for the universal curve over $\ovl{\R}_g$ in the spirit of \cite{HM}, or using the Hurwitz formula for the branched covering $\pi$.

\begin{theorem}
One has the following formula in $\mathrm{Pic}(\ovl{\R}_g)$:
$$K_{\ovl{\R}_g}=13\lambda-2(\delta_0^{'}+\delta_0^{''})-3\delta_0^{\mathrm{ram}}-2\sum_{i=1}^{[g/2]}
(\delta_i+\delta_{g-i}+\delta_{i:
g-i})-(\delta_1+\delta_{g-1}+\delta_{1: g-1}).$$
\end{theorem}
\begin{proof} We use the Harris-Mumford formula \cite{HM}
$$K_{\ovl{\M}_g}\equiv 13\lambda-2\delta_0-3\delta_1-2\delta_2-\cdots
-2\delta_{[g/2]},$$ together with the Hurwitz formula
for the ramified covering $\pi:\ovl{\R}_g\rightarrow \ovl{\M}_g$, which we recall, is simply branched along $\Delta_0^{\mathrm{ram}}$.  We find that
$K_{\ovl{\R}_g}=\pi^*(K_{\ovl{\M}_g})+\delta_0^{\mathrm{ram}}$.
\end{proof}

Without worrying for a moment about the singularities that the moduli space might have, to decide whether $\ovl{\R}_g$ has non-negative Kodaira dimension is equivalent to knowing whether there exist any Siegel modular forms on $\ovl{\R}_g$ of weight $13$ and vanishing with order $2$ or $3$ along infinity.

\section{The singularities of $\ovl{\R}_g$}

The Kodaira dimension $\kappa(X)$ of a complex normal projective variety $X$ is defined as $\kappa(X):=\kappa(X')$, where $\epsilon:X'\rightarrow X$ denotes an arbitrary resolution of singularities. In general, the fact that the canonical sheaf $\Oo_X(K_X)$ is big, does not imply that $X$ is of general type. One only has an inequality $\kappa(X)\leq \kappa(X, K_X)$, relating the Kodaira dimension of $X$ to the Kodaira-Iitaka dimension of its canonical linear series. To give a very simple minded example where equality fails to hold, let $C\subset \PP^2$ be a plane quartic curve with three nodes. Then by the adjunction formula, $K_C = \Oo_C(1)$, and this divisor is of course big. But $\kappa(C) = -\infty$ because the normalization of $C$ is a rational curve. The reason is that the singularities of $C$ impose too many adjunction conditions. In order to determine the Kodaira dimension of ${\ovl{\R}_{g}}$ by working directly with its canonical bundle $K_{\ovl{\R}_g}$ (and this is certainly what one wants, for a desingularization of $\ovl{\R}_g$ would a priori have no modular interpretation), one must have control over the singularities of the coarse moduli space.
\vskip 3pt

\ \ \ \ Just like in the case of $\mm_g$, the local structure of $\ovl{\R}_g$ is governed by Kodaira-Spencer deformation theory. Let $X$ be a quasi-stable curve of genus $g$, and denote by $\omega_X$ (respectively $\Omega_X$) the dualizing sheaf of $X$ (respectively the sheaf of K\"ahler differentials on $X$).
Note that $\omega_X$ is locally free, whereas $\Omega_X$ fails to be locally free at the nodes of $X$. There is a \emph{residue map}
$$\mathrm{res}:\omega_X\rightarrow \bigoplus_{p\in \mathrm{Sing}(X)} \mathbb C_p, \ \ \ \omega\mapsto \bigl(\mathrm{Res}_p(\omega)\bigr)_{p\in \mathrm{Sing}(X)}, $$
which is well-defined because the residues of a $1$-form $\omega\in H^0(X, \omega_X)$ along the two branches of $X$ corresponding to a node $p\in \mathrm{Sing}(X)$ coincide. There exists an exact sequence
$$\Omega_X\longrightarrow \omega_X\stackrel{\mathrm{res}}\longrightarrow \bigoplus_{p\in \mathrm{Sing}(X)} \mathbb C_p\longrightarrow 0.$$
We also recall that an \'etale neighbourhood of  $[C]\in \mm_g$ is given by a neighbourhood of the origin in the quotient
$$T_{[C]}(\mm_g)=\mbox{Ext}^1_C (\Omega_C, \Oo_C)/\mathrm{Aut}(C)=\bigl(H^0(C, \omega_C\otimes \Omega_C)\bigr)^{\vee}/\mathrm{Aut}(C).$$
One has a similar local description of $\ovl{\R}_g$. First of all, note that the versal deformation space of a Prym curve $[X, \eta, \beta]$ coincides with that of its stable model. The concept of an automorphism of a Prym curve has to be defined with some care:

\begin{definition}
An \emph{automorphism} of a Prym curve $[X, \eta, \beta]\in \ovl{\R}_g$  is an automorphism
$\sigma\in \mathrm{Aut}(X)$ such that there exists an isomorphism of
sheaves $\gamma:\sigma^*\eta\rightarrow \eta$ making the
following diagram commutative.
\begin{eqnarray*}
\xymatrix @!0 @R=1.3cm @C=2cm {
        (\sigma^*\eta)^{\otimes  2}  \ar[r]^-{  \gamma^{\otimes 2}}   \ar[d]_{\sigma^* \beta} &
        \eta^{\otimes 2} \ar[d]^{\beta}\\
        \sigma^* \mathcal{O}_X  \ar[r]^{\simeq} & \mathcal{O}_X
        }
\end{eqnarray*}
\end{definition}
If $C:=\mbox{st}(X)$ denotes the stable model of $X$ obtained by contracting all exceptional components of $X$, then there is a group
homomorphism $\mbox{Aut}(X, \eta, \beta) \rightarrow \mbox{Aut}(C)$ given by
$\sigma \mapsto \sigma_{C}$.  We call a node $p\in \mathrm{Sing}(C)$ \emph{exceptional} if it corresponds to an exceptional component that gets contracted under the map $X\rightarrow C$.
\vskip 3pt

\ \ \ \ We fix a Prym curve $[X,\eta,\beta] \in \ovl{\R}_{g}$.
An \'etale neighbourhood of $[X, \eta, \beta]$ is
isomorphic to the quotient of the versal deformation space
$\mathbb C_{\tau}^{3g-3}$ of $[X, \eta, \beta]$ modulo the action of the automorphism group
$\mathrm{Aut}(X, \eta, \beta)$. If $\mathbb C^{3g-3}_t=\mbox{Ext}^1(\Omega_C^1, \Oo_C)$ denotes
the versal deformation space of $C$, then the map
$\mathbb C_\tau^{3g-3} \rightarrow \mathbb C_t^{3g-3}$ is given by $t_i=\tau_i^2$, if
$(t_i=0) \subset \mathbb C_t^{3g-3}$ is the locus where an exceptional node
$p_i$ persists and $t_i=\tau_i$ otherwise. The
 morphism $\pi:\ovl{\R}_g\rightarrow\mm_g$ is given locally by the map
$$\mathbb C_{\tau}^{3g-3}/\mathrm{Aut}(X, \eta, \beta)\rightarrow \mathbb C_t^{3g-3}/\mathrm{Aut}(C).$$

This discussion illustrates the fact that $\ovl{\R}_g$ is a space with finite quotient singularities. It is a basic question to describe canonical  finite quotient singularities and the answer is provided by the \emph{Reid-Shepherd-Barron-Tai criterion} \cite{Re}.

\begin{definition}\label{can}
A $\mathbb Q$-factorial normal projective variety $X$ is said to have \emph{canonical singularities} if for any sufficiently divisible integer $r\geq 1$ and for a resolution of singularities $\epsilon:X'\rightarrow X$, one has that $\epsilon_*(\omega_{X'}^{\otimes r})=\Oo_X(rK_X)$. If this property is satisfied in a neighbourhood of a point $p\in X$, one says that $X$ has a canonical singularity at $p$.
\end{definition}

From the definition it follows that a section $s$ of $\Oo_X(rK_X)$ regular around $p\in X$ extends regularly to a neighbourhood of $\epsilon^{-1}(p)$. Canonical singularities appear in the Minimal Model Program as the singularities of canonical models of varieties of general type.
\vskip 3pt

\ \ \ \ Assume now $V:=\mathbb C^m$ and let $G\subset GL(V)$ be a finite group. We fix an element $g\in G$ with $\mbox{ord}(g)=n$.
The matrix corresponding to the action of $g$ is conjugate to a diagonal matrix $\mathrm{diag}(\zeta^{a_1},\ldots,\zeta^{a_{m}})$,  where $\zeta$ is an $n$-th root of unity and $0\leq a_i<n$ for $i=1, \ldots, m$. One  defines the \emph{age} of $g$ as the following sum
\vskip 4pt
\begin{center}
\fbox{$\mathrm{age}(g) := \frac{a_1}{n}+\cdots + \frac{a_{m}}{n}.$}
\end{center}
\begin{definition}
The element $g\in G$ is said to be \emph{junior} if $\mathrm{age}(g) < 1$ and \emph{senior} otherwise.
\end{definition}

We have the following characterization \cite{Re} of finite quotient canonical singularities:
\begin{theorem}\label{rsbt}
Let $G\subset GL(V)$ be a finite subgroup acting without quasi-reflections. Then the quotient $V/G$ has canonical singularities if and only if each non-trivial element $g\in G$ is senior.
\end{theorem}
\begin{remark} If $g\in G$ acts as a quasi-reflection, then $\{v\in V: g\cdot v=v\}$ is a hyperplane and
$g\sim \mathrm{diag}(\zeta^{a_1}, 1, \ldots, 1)$, hence $\mathrm{age}(g)=\frac{a_1}{n}<1$, that is, each quasi-reflection is junior. On the other hand, obviously quasi-reflections do not lead to a singularities of $V/G$, which is the reason for their exclusion from the statement of the Reid-Shepherd-Barron criterion.
\end{remark}

\ \ \ \ How does one apply Theorem \ref{rsbt} to study the singularities of $\ovl{\R}_g$? Let us fix a point $[X, \eta, \beta]\in \ovl{\R}_g$ as well as the \'etale neighbourhood $\mathbb C_{\tau}^{3g-3}$ defined above. We denote by $H\subset \mathrm{Aut}(X, \eta, \beta)$ the subgroup generated by automorphism acting as quasi-reflections on $\mathbb C_{\tau}^{3g-3}$. The quotient map
$$\mathbb C_{\tau}^{3g-3}\rightarrow \mathbb C_{\tau}^{3g-3}/H:=\mathbb C_{v}^{3g-3}$$
is given by $v_i:=\tau_i^2$ if the coordinate $\tau_i$ corresponds to smoothing out an elliptic tail of $X$ and $v_i:=\tau_i$ otherwise. By definition, $\mathrm{Aut}(X, \eta, \beta)$ acts on $\mathbb C_{v}^{3g-3}$ without quasi-reflections, hence by applying Theorem \ref{rsbt}, the quotient $\ovl{\R}_g$ has a canonical singularity at $[X, \eta, \beta]$ if an only if each automorphism is senior. In that is the case, forms defined in a neighbourhood of $[X, \eta, \beta]$ extend locally to any resolution of singularities. Unfortunately, $\ovl{\R}_g$ \emph{does have} non-canonical singularities as the following simple example demonstrates:

\begin{example} Let us choose an elliptic curve $[C_1, p]\in \M_{1, 1}$ with $\mbox{Aut}(C_1, p)=\mathbb Z_6$, as well as an arbitrary pointed curve $[C_2, p]\in \M_{g-1, 1}$ together with a non-trivial point of order two $\eta_{2}\in \mathrm{Pic}^0(C_2)-\{\Oo_{C_2}\}$. We consider a stable Prym curve $$[X:=C_1\cup_p C_2, \eta]\in \ovl{\R}_g,$$ where $\eta_{C_1}=\Oo_{C_1}$ and $\eta_{C_2}=\eta_2$. We consider an automorphism
$\sigma \in \mathrm{Aut}(X, \eta, \beta)$, where $\sigma_{C_2}$ is trivial and $\sigma_{C_1}\in \mathrm{Aut}(C_1)$ generates $\mathrm{Aut}(C_1)$. In the versal deformation space $\mathbb C_{\tau}^{3g-3}$ there exists two coordinates $\tau_1$ and $\tau_2$ corresponding to directions which preserve the node $p\in X$ and deform the $j$-invariant respectively. One can find a $6$-th root of unity $\zeta_6$ such that the action of $\sigma$ on $\mathbb C_{\tau}^{3g-3}$ is given by:
$$\sigma\cdot \tau_1=\zeta_6\tau_1, \ \ \sigma\cdot \tau_2=\zeta_6^2 \tau_2 \  \mbox{ and } \sigma\cdot \tau_i=\tau_i, \ \mbox{ for } i=3, \ldots, 3g-3.$$
The quotient map $\mathbb C_{\tau}^{3g-3}\rightarrow \mathbb C_v^{3g-3}$ is given by the formulas:
$$v_1=\tau_1^2,\ \ v_2=\tau_2 \ \mbox{ and } \ \ v_i=\tau_i \mbox{ for } i=3, \ldots, 3g-3.$$
Therefore the action of $\sigma$ on $\mathbb C_v^{3g-3}$ can be summarized as follows:
$$\sigma\cdot v_1=\zeta_6^2 v_1,\ \sigma\cdot v_2=\zeta_6^2 v_2 \mbox{ and } \sigma\cdot v_i=v_i \mbox{ for } i=3, \ldots, 3g-3.$$
Therefore $\mbox{age}(\sigma)=\frac{2}{6}+\frac{2}{6}=\frac{2}{3}<1$, and this leads  to a non-canonical singularity.
\end{example}

\vskip 3pt
The good news is that, in some sense, this is the only source of examples of non-canonical singularities. By a detailed case by case analysis, one proves the following characterization \cite{FL} Theorem 6.7 of the locus of non-canonical singularities:

\begin{theorem} Set $g\geq 4$. A point $[X, \eta, \beta]\in \ovl{\R}_g$ is a non-canonical singularity if and only if $X$ possesses an elliptic subcurve $C_1\subset X$ with $|C_1\cap \overline{(X-C_1)}|=1$, such that the $j$-invariant of $C_1$ is equal to zero, and the restriction $\eta_{C_1}$ is trivial.
\end{theorem}

Therefore $\ovl{\R}_g$ has a codimension two locus of non-canonical singularities, but the key fact is, that this locus of relatively simple and can be easily resolved. Even though there are \emph{local obstructions} to lifting pluri-canonical forms from $\ovl{\R}_g$,  Theorem 6.1 from \cite{FL}  shows that these are not \emph{global obstructions}
and in particular the Kodaira dimension of $\ovl{\R}_g$ equals the Kodaira-Iitaka dimension of the canonical linear series $|K_{\ovl{\R}_g}|$. This theorem is also a generalization of the result of Harris and Mumford who treated the case of $\mm_g$:

\begin{theorem}\label{extension}
Let us fix $g \geq 4$ and $\eps \::\: \widetilde{\R}_{g} \ra \ovl{\R}_{g}$  a resolution of singularities. Then for every integer $n\geq 1$ there is an isomorphism
$$\eps^*:H^0(\widetilde{\R}_{g}, K_{\widetilde{\R}_{g}}^{\otimes n}) \simeq H^0(\ovl{\R}_{g}, K_{\ovl{\R}_{g}}^{\otimes n}).$$
\end{theorem}

\ \ \ To sum up these considerations, $\kappa(\ovl{\R}_g)$ equals the Iitaka dimension of the canonical linear series. For all questions concerning birational classification, $\ovl{\R}_g$ is as good as a smooth variety.

\section{Geometric cycles on $\ovl{\R}_{g}$}

For every normal $\mathbb Q$-factorial variety $X$ for which an extension result along the lines of Theorem \ref{extension} holds, in order to show that $\kappa(X)\geq 0$ is suffice to prove that $K_X$ is an effective class. Following a well-known approach pioneered by Harris and Mumford \cite{HM} in the course of their proof  that the moduli space of curves $\mm_g$ is of general type for $g\geq 24$, one could attempt to construct explicitly sections of the pluri-canonical bundle on $\ovl{\R}_g$ by means of algebraic geometry, by considering geometric conditions on Prym curves that fail along a hypersurface in the moduli space $\R_g$. Such geometric conditions must be amenable to degeneration to stable Prym curves, for one must be able to compute the class in $\mathrm{Pic}(\ovl{\R}_g)$ of the closure of the locus where the condition fails. In particular, points in the boundary must have a strong geometric characterization. Finally, one must recognize an empirical geometric principle that enables to distinguish between divisorial geometric conditions that are likely to lead to divisors of small slope on $\ovl{\R}_g$ (ideally to extremal points in the effective cone of divisors
$\mathrm{Eff}(\ovl{\R}_g)$) and divisors of high slope which are less interesting. For instance in the case of $\mm_g$, one is lead to consider only divisors containing the locus of curves that lie on $K3$ surfaces, see \cite{FP}:

\begin{proposition} Let $D\subset \mm_g$ be any effective divisor. If the following slope inequality $s(D)<6+12/(g+1)$ holds, then $D$ must contain
the locus $$\mathcal{K}_g:=\{[C]\in \mathcal{M}_g: C\ \mbox{ lies on a } K3 \mbox{ surface}\}.$$
\end{proposition}

This of course sets serious geometric constraints of the type of divisors on $\mm_g$ whose class is worth computing, since it is well-known that curves of $K3$ surfaces behave generically from many points of view (e.g. Brill-Noether theory). Here in contrast we are looking for geometric conditions with respect to which the $K3$ locus behaves non-generically. We refer to \cite{F1} for a way to produce systematically divisors on $\mm_g$ having slope less than $6+12/(g+1)$. We close this introductory discussion by summarizing the numerical conditions that an effective divisor on $\ovl{\R}_g$ ought to satisfy, in order to show that the moduli space has maximal Kodaira dimension. Precisely,  $\ovl{\R}_g$ is of general type, if there exists a divisor $D\subset \ovl{\R}_g$ such that
 $$D\equiv
a\lambda-b_0^{'}\delta_0^{'}-b_0^{''}\delta_{0}^{''}-b_0^{\mathrm{ram}}\delta_{0}^{\mathrm{ram}}
-\sum_{i=1}^{[g/2]} (b_i\delta_i+b_{g-i}\delta_{g-i}+b_{i:
g-i}\delta_{i: g-i})\ \in \mathrm{Eff}(\ovl{\R}_g),$$ satisfying the
following inequalities: \begin{equation}\label{inequ}
\mathrm{max}\Bigl\{\frac{a}{b_0^{'}}, \
\frac{a}{b_0^{''}}\Bigr\}<\frac{13}{2}, \ \ \mbox{  }
\mathrm{max}\Bigl\{\frac{a}{b_0^{\mathrm{ram}}},\  \frac{a}{b_1},
\frac{a}{b_{g-1}}, \  \frac{a}{b_{1: g-1}}\Bigr\}<\frac{13}{3}
\end{equation}
and
$$
\mathrm{max}_{i\geq 1}\bigl\{\frac{a}{b_i}, \ \frac{a}{b_{g-i}},
\frac{a}{b_{i: g-i}}\bigr\}<\frac{13}{2}.
$$
It is explained in \cite{F2} how one can rederive the results of \cite{HM} using Koszul divisors on $\mm_g$, and how more generally, loci in moduli given in terms of syzygies of the objects they parametrize, lead to interesting geometry on moduli spaces.  It is thus natural to try to use the same approach
in the case of $\ovl{\R}_{g}$, with the role of the canonical curve $C\stackrel{|K_C|}\longrightarrow \PP^{g-1}$ being played by the
\emph{Prym-canonical curve}  $C\stackrel{|K_C\otimes \eta|}\longrightarrow \PP^{g-2}$.

\vskip 3pt

\ \ \ \ Let us fix a Prym curve $[C,\eta] \in \R_{g}$ and the Prym-canonical line bundle $L:=K_C\otimes \eta \in W^{g-2}_{2g-2}(C)$ inducing a morphism
$$\phi_{L} \::\: C \rightarrow \PP^{g-2}.$$
We denote by $I(L)\subset S:=\mathbb C[x_0, \ldots, x_{g-2}]$ the ideal of the Prym-canonical curve and  consider the minimal resolution of the homogeneous coordinate ring $S(L):=S/I(L)$ by free graded $S$-modules:
$$\cdots \rightarrow F_i\rightarrow \cdots F_2\rightarrow F_1\rightarrow F_0\rightarrow S(L)\rightarrow 0,$$
where $F_i=\bigoplus_j S(-i-j)^{b_{i, j}(C, L)}$. The numbers $b_{i, j}(C, L)=\mbox{dim}_{\mathbb C} \mathrm{Tor}_i^{i+j}\bigl(S(L), \mathbb C\bigr)$ are the \emph{graded Betti numbers}
of the pair $(C, L)$ and encode the number of $i$-th order syzygies of degree $j$ in the equations of the Prym-canonical curve.
The graded Betti numbers can be computed via Koszul cohomology, using the resolution of the ground field $k:=\mathbb C$ by
free graded $S$-modules. Precisely, we write the complex
$$\ldots \xra{} \bigwedge^{i+1} H^0(C, L)\otimes H^0(C, L^{\otimes (j-1)}) \xra{d_{i+1, j-1}} \bigwedge^{i} H^0(C, L)\otimes H^0(C, L^{\otimes j})$$ $$\xra{d_{i,j}} \bigwedge^{i-1} H^0(C, L)\otimes H^0(C, L^{\otimes (j+1)}) \xra{ }\ldots,$$
where $$d_{i, j}(f_1\wedge \ldots \wedge f_i\otimes u):=\sum_{l=1}^i (-1)^l f_1\wedge \ldots \wedge \hat{f_l}\wedge \ldots \wedge f_i\otimes (f_l u),$$ is the Koszul differential, with  $f_1, \ldots, f_i\in H^0(C, L)$ and $u\in H^0(C, L^{\otimes i})$. One easily checks that $d_{i, j}\circ d_{i+1, j-1}=0$, and defines the \emph{Koszul cohomology groups}
\begin{center}
\fbox{
$K_{i,j}(C,L):=\mathrm{Ker}\ d_{i,j}/\mathrm{Im}\  d_{i+1,j-1}.$ }
\end{center}
Then $\mbox{dim } K_{i, j}(C, L)=b_{i, j}(C, L)$. The Koszul cohomology theory has been introduced by M. Green \cite{Gr} and can be seen as a highly effective way of packaging geometrically the algebraic information contained in the homogeneous coordinate ring of an embedded variety.

\ \ \ \ We consider the locus in $\R_g$ consisting of Prym curves having a non-linear $i$-th syzygy, that is,
 $$\U_{g, i}:=\Bigl\{[C,\eta]\in\R_{g} \: : \: K_{i,2}(C,K_C\otimes \eta)\neq 0\Bigr\}.$$ In order to determine the expected dimension
of $\U_{g, i}$ as a degeneracy locus inside $\R_g$, we find a global determinantal presentation of $\U_{g, i}$.
Using a standard argument involving the \emph{Lazarsfeld bundle} $M_L$ defined via the following exact sequence on $C$
$$0\longrightarrow M_L\longrightarrow H^0(C, L)\otimes \Oo_C\longrightarrow L\longrightarrow 0,$$
one has the following identification, see e.g. \cite{GL} Lemma 1.10:
\begin{equation}\label{koszul}
 K_{i, 2}(C, L)=\frac{H^0(C, \wedge^i
M_L\otimes L^{\otimes 2})}{\mathrm{Im} \{\wedge^{i+1} H^0(C,
L)\otimes H^0(C, L)\}}\ .
\end{equation}
After some diagram chasing explain for instance in detail in \cite{F2}, one obtains that $K_{i, 2}(C, L)\neq 0$ if and only if $H^1(C, \wedge^{i+1}
M_L\otimes L)\neq 0$.  After even more manipulations, this condition is equivalent to requiring that the restriction map
\begin{equation}\label{ressyz}
\varphi{[C, \eta]}:H^0\bigl(\PP^{g-2}, \wedge^i M_{\PP^{g-2}}\otimes \Oo_{\PP^{g-2}}(2)\bigr)\longrightarrow H^0\bigl(C, \wedge^i M_L\otimes L^{\otimes 2}\bigr)
\end{equation}
have a kernel of dimension at least $$\mathrm{dim}\ \mathrm{Ker}(\varphi{[C, \eta]})\geq {g-3\choose i}\frac{(g-1)(g-2i-6)}{i+2}.$$ We refer to  \cite{FL} Section 3 for full details. We point out that the dimension of both vector spaces that enter the map $\varphi [C, \eta]$ remain constant as $[C, \eta]$ varies in moduli, precisely
$$h^0\bigl(\PP^{g-2}, \wedge^i M_{\PP^{g-2}}(2)\bigr)=(i+1){g \choose i+2}$$
and
$$h^0(C, \wedge^i M_L\otimes L^2)=\chi(C, \wedge^i M_L\otimes L^{\otimes 2})={g-2\choose i}\Bigl(-\frac{i(2g-2)}{g-2}+3(g-1)\Bigr).$$ We get a divisorial condition in moduli,  exactly when the vector spaces in (\ref{ressyz}) have the same dimension, and the required geometric condition is that the map $\varphi{[C, \eta]}$ be an isomorphism. This happens precisely when $g=2i+6$.

\begin{proposition}
 Set $g:=2i+6$. There exist vector bundles $\A$ and $\B$ on $\R_{2i+6}$ with  $\rk(\A) = \rk(\B)$ as well as a vector bundle morphism $\varphi \::\: \A \ra \B$ such that $\U_{g, i}$ is exactly the degeneracy locus of $\phi$.
In other words, $\U_{2i+6, i}$ is a virtual divisor on $\R_{2i+6}$.
\end{proposition}

By analogy with the case of the classical Green's Conjecture, it is reasonable to conjecture that the morphism $\varphi:\A\rightarrow \B$ is generically non-degenerate, and then $\U_{g, i}$ is a genuine divisor on $\R_g$. We recall  the statement of the \emph{Prym-Green Conjecture}
\cite{FL} Conjecture 0.7:

\begin{conjecture}\label{prymgreen}
For a general  curve $[C,\eta] \in \R_{2i+6}$ one has the vanishing
\begin{center}
\fbox{ $K_{i,2}(C,K_C\otimes \eta) =0.$ }
\end{center}
\end{conjecture}

Note that if true, the Prym-Green Conjecture is sharp. For $g<2i+6$ it follows from previous considerations that $K_{i, 2}(C, K_C\otimes \eta)\neq 0$ for any $[C, \eta]\in \R_g$.

\begin{example} We explain the simplest case of the Prym-Green Conjecture, namely when $i=0$ and $g=6$. Then one has an identification
$$\U_{6, 0}=\{[C, \eta]\in \R_{6}: K_{0, 2}(C, K_C\otimes \eta)\neq 0\}=$$
$$=\bigl \{[C, \eta]\in \R_6: \mu_0(K_C\otimes \eta):\mathrm{Sym}^2 H^0(C, K_C\otimes \eta)\stackrel{\neq}\longrightarrow H^0(C, K_C^{\otimes 2})\bigr\}.$$
Observe that via Kodaira-Spencer theory, the following identifications hold
$$T_{[C, \eta]}(\R_6)=T_{[C]}(\M_6)=\bigl(H^0(C, K_C^{\otimes 2})\bigr)^{\vee}$$ and
$$T_{\mathrm{Pr}_6 [C, \eta]}(\A_5)=\bigl(\mathrm{Sym}^2 H^0(C, K_C\otimes \eta)\bigr)^{\vee},$$ that is, the multiplication map $\mu_0(K_C\otimes \eta)$ is the codifferential of the Prym map and  $\U_{6, 0}$ is the ramification divisor of the generically finite covering $\mathrm{Pr}_6: \R_6\rightarrow \A_5$. The Prym-Green Conjecture in genus $6$ is equivalent to the infinitesimal Prym-Torelli Theorem!
An example of a Prym curve $[C, \eta]\in \R_6$ for which $\mu_0(K_C\otimes \eta)$ is an isomorphism, that is, $[C, \eta]\in \R_6-\U_{6, 0}$, is provided by Beauville \cite{B3}. Let $C\subset \PP^2$ be a smooth plane quintic and choose a quartic $X\subset \PP^2$ everywhere tangent to $C$, that is,
$X\cdot C=2(p_1+\cdots+p_{10})$, where $p_1, \ldots, p_{10}\in C$. Then take $\eta:=\Oo_C(2)(-p_1-\cdots-p_{10})$, thus $[C, \eta]\in \R_6$. It is not difficult to verify directly that the resulting Prym-canonical curve $\phi_{K_C\otimes \eta}:C\hookrightarrow \PP^4$ does not lie on a quadric.
\end{example}

\begin{example}
As a consequence of the Green-Lazarsfeld non-vanishing theorem \cite{GL1}, one can exhibit two codimension two loci in $\R_{g}$ contained in $\U_{g, i}$, namely
$$\mathcal{Z}_1:=\pi^*(\M_{g, i+3}^1)=\{[C, \eta]\in \R_g: \mathrm{gon}(C)\leq i+3\}$$
and
$$\mathcal{Z}_2:=\{[C, \eta]\in \R_g: \eta \in C_{i+2}-C_{i+2}\subset \mathrm{Pic}^0(C)\}.$$
It is a very interesting open problem to find a codimension one subvariety of $\R_g$ which contains both $\mathcal{Z}_1$ and $\mathcal{Z}_2$ and might be a suitable candidate to be equal to $\U_{g, i}$. We envisage here a geometric condition \emph{in terms of Prym varieties} which holds in codimension one in the moduli space, and which a posteriori, should be equivalent to the syzygy condition $K_{i, 2}(C, K_C\otimes \eta)\neq 0$. For $i=0$ we have seen that this condition is simply that the differential of the Prym map be not bijective.
\end{example}

\vskip 4pt
\ \ \ \ The Prym-Green Conjecture is known to hold in bounded degree. Note that for any integer $l\geq 3$ one can formulate an analogous \emph{level $l$ Prym-Green Conjecture} predicting the vanishing $$K_{i, 2}(C, K_C\otimes \xi)=0,$$ where $\xi\in \mathrm{Pic}^0(C)-\{\Oo_C\}$ satisfies $\xi^{\otimes l}=\Oo_C$, with $C$ being a general curve of genus $2i+6$.

\subsection{Koszul divisor calculations on $\ovl{\R}_g$.} Independent of the validity of the Prym-Green Conjecture, one could try to compute the virtual class of a compactification of $\U_{g, i}$. It is shown in \cite{FL} that over a partial compactification $\R_g \subset \widetilde{\R}_g \subset \ovl{\R}_g$ such that $\mbox{codim}(\ovl{\R}_g-\widetilde{\R}_g, \ovl{\R}_g)\geq 2$, there exist extensions $\widetilde{\A}$ and $\widetilde{\B}$ of the vector bundles $\A$ and $\B$ as well as a homomorphism  denoted by $\widetilde{\varphi}:\widetilde{\A}\rightarrow \widetilde{\B}$ such that the degeneracy locus of $\widetilde{\varphi}$ is precisely the closure of $\U_{g, i}$ inside $\widetilde{\R}_g$. Furthermore, the vector bundles $\widetilde{\A}$ and $\widetilde{\B}$ have modular meaning and one can compute their Chern classes in terms of tautological classes:

\begin{theorem}\label{classprymgreen}
Set $g= 2i+6$. We have the following formula for the virtual class of the Prym-Green degeneracy locus:
$$\bigl[\overline{\U}_{g, i}^{\mathrm{virt}}\bigr] = {2i+2 \choose i}\Bigl(\frac{3(2i+7)}{i+3}\lambda - (\delta_0^{'}+\delta_0^{''}) - \frac{3}{2}\delta_0^{\mathrm{ram}}-\dots\Bigr)\in \mathrm{Pic}(\ovl{\R}_g).$$
\end{theorem}

It is instructive to compare $\bigl[\overline{\U}_{g, i}^{\mathrm{virt}}\bigr]$ against the formula of the canonical class:
$$K_{\ovl{\R}_g}\equiv 13\lambda-2(\delta_0^{'}+\delta_0^{''})-3\delta_0^{\mathrm{ram}}-\cdots \in \mathrm{Pic}(\ovl{\R}_g).$$
Assuming the Prym-Green Conjecture in genus $g$, so that $\overline{\U}_{g, i}$ is a genuine divisor on $\ovl{\R}_g$ as opposed to a virtual one, one obtains that the class $K_{\ovl{\R}_g}$ is big precisely when the following equality is satisfied
$$\frac{3(2i+7)}{i+3}<\frac{13}{2}\Leftrightarrow i\geq 3.$$
When $g\geq 22$ it is known that $\mm_g$ is of general type, see \cite{HM}, \cite{EH}, \cite{F2}. This implies that $\ovl{\R}_g$, as a branched covering of $\mm_g$, is of general type as well. Even though the validity of the Prym-Green Conjecture for arbitrary $g=2i+6$ remains a challenging open problem, for applications to the birational geometry of $\ovl{\R}_g$ it is enough to know that the conjecture holds in bounded even genus $g\leq 20$. This is something that can be checked (with quite some effort!) by degeneration with the help of the computer algebra program Macaulay2. To summarize we have the following result \cite{FL}:
\begin{theorem} The moduli space of stable Prym curves $\ovl{\R}_{2i+6}$ is a variety of general type for $i\geq 4$. The Kodaira dimension of
$\ovl{\R}_{12}$ is non-negative.
\end{theorem}

\subsection{Prym curves and the universal difference variety} The problem of determining the Kodaira dimension of $\ovl{\R}_g$ for odd genus has a relatively simpler solution that the even genus case. We follow \cite{FL} Section 2. We fix a smoth non-hyperelliptic curve $C$ of genus $g$. The \emph{$i$-th difference variety} of $C$ is defined as the image of
the difference map
$$\phi:C_i\times C_i \rightarrow \mbox{Pic}^0(C), \mbox{ } \ \phi(D_1,
D_2):=\Oo_C(D_1-D_2).$$

It is easy to prove, see e.g. \cite{ACGH}, that for $i<g/2$ the map $\phi$ is birational onto its image. The following definition is due to Raynaud \cite{R}:
\begin{definition} Let $E\in \mathcal{SU}_C(r, d)$ be a semistable vector bundle on a curve $C$, such that the slope $\mu:=d/r\in \mathbb Z$. The \emph{theta-divisor} of $E$ is defined as the non-vanishing locus
\begin{center}
\fbox{$\Theta_E:=\{\xi\in \mathrm{Pic}^{g-\mu-1}(C): H^0(C, E\otimes \xi)\neq 0\}.$}
\end{center}
\end{definition}
The locus $\Theta_E$ is a virtual divisor inside $\mathrm{Pic}^{g-\mu-1}(C)$, that is, it is either the full Picard variety when $H^0(C, E\otimes \xi)\neq 0$ for every $\xi$, or a genuine divisor when there exists a line bundle $\xi\in \mathrm{Pic}^{g-\mu-1}(C)$ such that $H^0(C, E\otimes \xi)=0$. In that case, $[\Theta_E]=r\theta$, where $\theta\in H^2(\mathrm{Pic}^{g-\mu-1}(C), \mathbb Z)$ is the class of the "classical" theta divisor. In the latter case, one says that $E$ possesses a theta divisor.
\vskip 3pt

\ \ \ \ Let us assume that $g:=2i+1$, therefore $C_i-C_i\subset \mathrm{Pic}^0(C)$ is a divisor. We denote by $Q_C:=M_{K_C}^{\vee}$ the dual of the Lazarsfeld bundle, therefore $\mu(Q_C)=2\in \mathbb Z$ and one may ask whether $Q_C$ and all its exterior powers have theta divisors, and if so, whether they have an intrinsic interpretation in terms of the geometry of the canonical curve.
Using a filtration argument due to Lazarsfeld,  one
finds that for a generic choice of distinct points $x_1, \ldots,
x_{g-2}\in C$, there is an exact sequence
$$0\longrightarrow \bigoplus_{l=1}^{g-2} \Oo_C(x_l)\longrightarrow
Q_C\longrightarrow K_C\otimes \Oo_C(-x_1-\cdots
-x_{g-2})\longrightarrow 0.$$
\ \ \ \ This leads to an inclusion of cycles
$C_i-C_i\subset \Theta_{\wedge^i Q_C}$.

The main result from \cite{FMP} states that for any smooth curve $[C]\in
\M_g$ the Raynaud locus $\Theta_{\wedge^i Q_C}$ is a divisor in
$\mbox{Pic}^0(C)$ (that is, $\wedge^i Q_C$ has a theta divisor), and
one has the following equality of cycles
\vskip 3pt
\begin{center}
\fbox{$\Theta_{\wedge^i Q_{C}}=C_i-C_i\subset
\mbox{Pic}^0(C).$}
\end{center}
\vskip 3pt
This identification shows that via the difference map, $C_i\times C_i$ is a resolution of singularities of
$\Theta_{\wedge^i Q_C}$.

Having produced a distinguished divisor in the degree zero Jacobian of each curve, we can use it to obtain codimension $1$ conditions in $\R_g$ by requiring that the point of order $2$ belong to this divisor.  We define the following locus in $\R_g$:
$$\mathcal{D}_{2i+1}:=\Bigl\{[C, \eta]\in \R_{2i+1}: \eta\in C_i-C_i\Bigr\}=$$
$$=\Bigl\{[C, \eta]\in \R_{2i+1}: H^0(C, \wedge^i Q_C\otimes \eta)\neq 0\Bigr\}.$$

Note that $\mathcal{D}_{2i+1}$ has two incarnations, the first one of a more geometric nature showing that points $[C, \eta]\in \mathcal{D}_{2i+1}$ are characterized by the existence of a certain secant to the Prym-canonical curve $C\stackrel{|K_C\otimes \eta|}\longrightarrow \PP^{2i-1}$, the second of a determinantal nature which is very useful if one wishes to compute the class of the closure $\overline{\mathcal{D}}_{2i+1}$ of $\mathcal{D}_{2i+1}$ inside $\ovl{\R}_{2i+1}$. One has the following formula, see \cite{FL} Theorem 0.2:

\begin{theorem}\label{hyper}
The class of the closure $\overline{\mathcal{D}}_{2i+1}$ inside
$\ovl{\R}_{2i+1}$ is equal to:
$$\overline{\mathcal{D}}_{2i+1}\equiv \frac{1}{2i-1}{2i\choose
i}\Bigl((3i+1)\lambda-\frac{i}{2}(\delta_0'+\delta_0^{''})-\frac{2i+1}{4}\delta_0^{\mathrm{ram}}
-\cdots\Bigr)\in \mathrm{Pic}(\ovl{\R}_{2i+1}).$$
\end{theorem}
Comparing this formula against the canonical class $K_{\ovl{\R}_g}$, we prove the following:
\begin{theorem}
The moduli space $\ovl{\R}_{2i+1}$ is of general type for $g\geq 7$.
\end{theorem}

\section{The birational geometry of the moduli of spin curves}

In this last lecture we propose to treat briefly the moduli space $\S_g$ classifying even theta-characteristics over curves of genus $g$, that is, the parameter space
$$\S_g^+:= \Bigl\{[C,\eta]: [C]\in \M_g, \ \eta \in \Pic^{g-1}(C),\  \eta^{\otimes 2} = K_C \ \mbox{ and }\ h^0(C, \eta) \equiv 0\ \mathrm{mod } \ 2\Bigr\}.$$  At first sight, one might think that the geometry of $\S_g^+$ should mirror rather closely that of $\R_g$, since both spaces parametrize curves with level two structures. Indeed there are certain similarities between $\S_g^+$ and $\R_g$. Both spaces are covers of $\M_g$ and they admit very similar compactifications via stable Prym and spin curves respectively. On the other hand, there are also important differences  reflected in
birational geometry (to put it loosely, $\S_g^+$ seems to be easier to describe than $\R_g$), as well as in the study of singularities
(and here in contrast, the singularities of $\ss_g^+$ appear to be more complicated than those of $\ovl{\R}_g$). Both spaces admit obvious higher level generalizations and one can talk of moduli spaces $\R_{g, l}$ and $\S_{g, l}^+$ for any level $l\geq 3$. We shall not discuss in these lectures the properties (or even the definition) of these spaces, but the trends observed for level $2$ (including Kodaira dimension and singularities) persists and become even more pronounced as the level $l$ increases.
\vskip 3pt
\ \ \ A geometrically meaningful compactification of $\S_g^+$ by means of stable spin curves, has been found by Cornalba \cite{C}:
 \begin{definition}
An \emph{even stable spin curve of genus $g$} is a triple $[X, \eta, \beta]$ where:
\begin{itemize}
\item $X$ is a quasi-stable curve with $p_a(X) = g$.
\item $\eta \in \Pic^{g-1}(X)$.
\item $\eta_E = \Oo_E(1)$ for all exceptional components $E\subseteq X$.
\item $\beta: \eta^{\otimes 2} \ra \omega_X$ is a sheaf morphism which is an isomorphism along each non-exceptional component of $X$.
\end{itemize}
\end{definition}

Hoping this shall not cause confusion with the previously discussed case of $\ovl{\R}_g$, we also denote by $\pi:\ss_g^+\rightarrow \mm_g$ the map given by $\pi([X, \eta, \beta]):=[\mathrm{st}(X)]$ forgetting the spin structure and contracting, if necessary, the exceptional components. Note that $\deg( \pi)=2^{g-1}(2^g+1)$ is the number of even theta-characteristics on any smooth curve of genus $g$.  One has the following commutative diagram:

$$\xymatrix{
  & \Sp \ar[d]_{\pi} \ar@{^{(}->}@<-0.5ex>[r] & \Spo \ar[d]^{\pi}& \\
  & \M_g \ar@{^{(}->}@<-0.5ex>[r] & \ovl{\M}_g & \\
}$$

The following is a complete birational classification of $\Spo$ in terms of Kodaira dimension.

\begin{theorem}\label{classi}\hfill
\begin{itemize}
\item \cite{FV2} The moduli space $\ss_g^+$ is uniruled for $g\leq 7$.
\item \cite{FV2} The Kodaira dimension of $\ss_8^+$ is equal to zero.
\item \cite{F3}  The moduli space $\Spo$ is of general type for $g\geq 9$.
\end{itemize}
\end{theorem}

\begin{remark}
We observe that as $g$ increases, $\ss_g^+$ becomes faster of general type than $\ovl{\R}_g$. The two spaces have different Kodaira dimension for genus $8$.
\end{remark}

\ \ \ \ It is instructive to repeat in the context of $\ss_g^+$ an exercise already carried out for $\ovl{\R}_g$ and determine in the process the ramification divisor  of the covering $\pi$:
\begin{example}\label{spinboundaries}
Let $\Delta_0 \subseteq \ovl{\M}_g$ denote the closure of the divisor of irreducible nodal curves. We choose a general point $[C_{xy}]\in\Delta_0$ and  an even stable spin curve $[X,\eta,\beta] \in \pi^{-1}([C_{xy}])$ with stable model $C_{xy}$. Then there are two possibilities, depending on whether $X$ contains an exceptional component or not:

$\bullet$ $X=C_{xy}$ and then $\eta \in \Pic^{g-1}(X)$. Denoting by $\eta_C\in \mathrm{Pic}^{g-1}(C)$ the pull-back of $\eta$ to the normalization of $X$, we observe that  $\eta^{\otimes 2}=K_C(x+y)$ and the fibers $\eta_C(x)$ and $\eta_C(y)$ can be identified in a unique way such that the resulting line bundle on $X$ satisfies  $h^0(X, \eta)\equiv 0\ \mbox{ mod }\  2$. We denote by $A_0$ the closure of the locus of such points in $\ss_g^+$.

$\bullet$ $X = C \cup_{\{x, y\}} E$, where $E\cong \PP^1$ is an exceptional component meeting the other component of $X$ in two points. Then by definition $\eta_E = \Oo_E(1)$ and an easy application of the Mayer-Vietoris sequence on $X$ gives that $\eta_C^{\otimes 2} = K_C$ with $h^0(C, \eta_C)\equiv 0 \ \mathrm{ mod }\  2$, that is, $\eta_C$ is an even theta-characteristic on $C$. The closure of such points $[X, \eta, \beta]\in \ss_g^+$ will be denoted by $B_0.$
\end{example}

\ \ \ Both $A_0, B_0$ are irreducible boundary divisors of $\ss_g^+$ and $\pi$ is simply branched over $B_0$.
Setting $\alpha_0:=[A_0]$ and $\beta_0:=[B_0]\in \mathrm{Pic}(\ss_g^+)$, the following formula holds:
\begin{equation}\label{irrpullback}
\pi^*(\Delta_0) = \alpha_0 + 2\beta_0.
\end{equation}
We leave it as an exercise to verify using Example \ref{spinboundaries} that indeed, $$\mbox{deg}(A_0/\Delta_0)+2\mbox{deg}(B_0/\Delta_0)=2^{g-1}(2^g+1).$$

For $1\leq i\leq [g/2]$, we denote by $A_i\subset \ss_g^+$ the closure of
the locus corresponding to pairs of even pointed spin curves $$\bigl([C,y, \eta_C],\ [D, y, \eta_D]\bigr)\in
\S_{i, 1}^+\times \S_{g-i, 1}^+$$ and by $B_i\subset \ss_g^+$ the
closure of the locus corresponding to pairs of odd spin curves $$\bigl([C, y, \eta_C], \ [D, y,
\eta_D]\bigr)\in \S_{i, 1}^-\times \S_{g-i, 1}^{-}.$$
Setting
$\alpha_i:=[A_i]\in \mathrm{Pic}(\ss_g^+), \beta_i:=[B_i]\in
\mathrm{Pic}(\ss_g^+)$, one has the relation
\begin{equation}\label{dipullback}
\pi^*(\delta_i)=\alpha_i+\beta_i.
\end{equation}
Again, we invite the reader to check that $\mathrm{deg}(A_i/\Delta_i)+\mathrm{deg}(B_i/\Delta_i)=2^{g-1}(2^g+1)$.
Applying the Riemann-Hurwitz formula to the covering $\pi$ coupled with formulas (\ref{irrpullback}) and (\ref{dipullback}), one obtains:
$$K_{\Spo} \equiv 13\lambda - 2\alpha_0 -3\beta_0 -3(\alpha_1+\beta_1)-2\sum_{i=2}^{[\frac{g}{2}]}(\alpha_i+\beta_i) \in \Pic(\Spo).$$

\subsection{The theta-null divisor}
Let $C\stackrel{|K_C|}\longrightarrow \PP^{g-1}$ be a non-hyperelliptic canonically embedded curve. The space of quadrics containing
$C$
$$I_2(C):=\mathrm{Ker}\bigl\{\mathrm{Sym}^2 H^0(C, K_C)\rightarrow H^0(C, K_C^{\otimes 2})\bigr\}$$
has dimension ${g-2\choose 2}$. The space of rank three quadrics inside $\mathrm{Sym}^2 H^0(C, K_C)$ also has codimension
${g-2\choose 2}$, therefore the condition that there exist a rank three quadric in $\PP I_2(C)$ is expected to be divisorial in moduli.
Let $Q\in I_2(C)$ be a rank three quadric, hence $\mathrm{Sing}(Q)$ is a $(g-4)$-dimensional linear space. Assume that $C\cap \mathrm{Sing}(Q)=\{x_1, \ldots, x_{n}\}$. Then the unique ruling of $Q$ cuts out a pencil $A$ of degree $g-1-\frac{n}{2}$ on $C$, such that
$K_C=A^{\otimes 2}\otimes \Oo_C(D)$. If $n=0$, that is, $C\cap \mathrm{Sing}(Q)=\emptyset$, then $A\in W^1_{g-1}(C)$ is a theta-characteristic, which prompts us to define the following subvariety of $\S_g^+$:
\begin{definition}
The \emph{theta-null divisor} on $\S_g^+$ is defined as the locus
$$\Theta_{\mathrm{null}} := \{[C,\eta] \in \S_g^+:  h^0(C, \eta) \geq 2\}.$$
\end{definition}
The locus in $\M_g$ consisting of curves whose canonical model lies on a rank three quadric breaks-up into components depending on the cardinality $\#(C\cap Q)$. For each integer $\frac{g+2}{2}\leq n\leq g-1$, we define the \emph{Gieseker-Petri divisor}
$$\mathcal{GP}_{g, k}^1:=\{[C]\in \M_g: \exists A\in W^1_k(C) \mbox{ such that } H^0(C, K_C\otimes A^{\otimes (-2)})\neq 0\}.$$
Note that $\pi(\Theta_{\mathrm{null}})=\mathcal{GP}_{g, g-1}$ and one has a set-theoretic equality
 $$\Bigl\{[C]\in \M_g: \mbox{there exists } \ Q\in \PP I_2(C) \mbox{ with } \ \mathrm{rank}(Q)\leq 3\Bigr\}=\bigcup_{k=[\frac{g+3}{2}]}^{g-1}\mathcal{GP}_{g, k}^1.$$
The following result is proved in \cite{F3} Theorem 0.2:

\begin{theorem}\label{thetnul}
The class of the closure of the locus of vanishing theta-characteristics in $\ss_g^+$ is equal to
$$\ovl{\Theta}_{\mathrm{null}} \equiv \frac{1}{4}\lambda - \frac{1}{16}\alpha_0 - \frac{1}{2}\sum_{i=1}^{[\frac{g}{2}]} \beta_i\in \mathrm{Pic}(\ss_g^+).$$
\end{theorem}

Quite remarkably, the formula in Theorem \ref{thetnul} contains no terms involving $\beta_0$ or $\alpha_i$ with $i>0$! We can compare this formula against $K_{\ss_g^+}$ and observe that $K_{\ss_g^+}$ is not expressible as a combination of $\Theta_{\text{null}}$ and boundary divisors and one need another effective divisor to offset the negative coefficient of $\beta_0$ in the expression of $K_{\ss_g^+}$.

\subsection{Brill-Noether divisors on $\ovl{\M}_g$}

The most classical divisors on $\mm_g$ are loci of curves carrying a certain linear series of type $\mathfrak g^r_d$. Let us fix integers  $r,d \geq 1$ such that the Brill-Noether number $$\rho(r,g,d) = g - (r+1)(g-d+r) = -1.$$
Recalling that $\rho(g, r, d)$ is the expected dimension of the determinantal subvariety $W^r_d(C)$ of the Jacobian $\mbox{Pic}^d(C)$ consisting of linear series of dimension at least $r$, see \cite{ACGH} Chapter 4, one expects  when $\rho(g, r, d)=-1$ the locus of curves with a $\mathfrak g^r_d$ to be a divisor.

Indeed, the subvariety $\Mgdr = \{[C]\in \M_g: W^r_d(C)\neq \emptyset\}$ is an irreducible divisor and the class of its compactification has been computed
\cite{EH}:
$$[\Mgdro] \equiv c_{g, r, d}\Bigl((g+3)\lambda - \frac{g+1}{6}\delta_0-\sum_{i=1}^{[\frac{g}{2}]} i(g-i)\delta_i\Bigr)\in \mathrm{Pic}(\mm_g).$$
After a  bit of linear algebra in the vector space $\mathrm{Pic}(\ss_g^+)$, one finds that there exist constants $a,b\in \Q_{> 0}$ such that $$a\ovl{\Theta}_{\text{null}}+b\pi^*(\Mgdro) \equiv \frac{11g+29}{g+1}\lambda - 2\alpha_0-3\beta_0-\sum_{i=1}^{[\frac{g}{2}]} \bigl(a_i\ \alpha_i+b_i\ \beta_i),$$
where $a_i, b_i\geq 2$ for $i\geq 2$ and $a_1, b_1\geq 3$.
Therefore $\Sp$ is of general type if $$\frac{11g+29}{g+1} < 13\Leftrightarrow g>8.$$

Theorem \ref{thetnul} also shows that $\ss_8^+$ cannot be uniruled, since we have found an explicit canonical divisor
 $$K_{\Soo} = a\Theta_{\text{null}}+b\pi^*(\ovl{\M}_{8,7}^2) + \sum_{i=1}^4(a_i\alpha_i + b_i\beta_i),$$
 where $\M_{8,7}^2 = \{[C]\in \M_8: W^2_7(C)\neq \emptyset\}$.
\vskip 3pt

\ \ \ \ To complete the proof of Theorem \ref{classi} and show that $K_{\Spo}$ is rigid, we use the following strategy and prove that the following statements hold, see \cite{FV2}:

\begin{itemize}
\item $\overline{\Theta}_{\mathrm{null}}$ is a uniruled extremal effective divisor.
\item The boundary divisors $A_i$ and $B_i$ where $i\geq 1$, as well as $\pi^*(\M_{8,7}^2)$ are extremal and rigid.
\item there exists a covering family of rational curves $R\subseteq \ovl{\Theta}_{\text{null}}$ such that $R\cdot \ovl{\Theta}_{\text{null}} < 0$, $R\cdot \pi^*(\ovl{\M}_{8,7}^2) = 0$ and $R\cdot \alpha_i = R\cdot \beta_i = 0$ for $1\leq i \leq 4$.
\end{itemize}

Assuming this, for all integers $n \geq 1$,  we can write $$|nK_{\Soo}| = |n(K_{\Soo} - a\ovl{\Theta}_{\text{null}})| + na\ovl{\Theta}_{\text{null}}.$$  We repeat this argument for the remaining divisors to get smaller and smaller linear systems, then we conclude that $\kappa(\ss_8^+)=0$.

\subsection{Mukai geometry of $\ovl{M}_8$}

Of the three conditions listed above, the last one is by far the most difficult to realize. The fact that it can be achieved is rather counter-intuitive. The curve $R\subset \ss_8^+$ on one hand, should be contained in $\overline{\mathrm{\Theta}}_{\mathrm{null}}$ therefore it should consist of Brill-Noether special spin curves. On the other hand, we require $\pi_*(R)$ be disjoint from $\mm_{8, 7}^2$, that is, $R$ should consist of spin curves which are general from the point of view of another Brill-Noether theoretic condition. The fact that such an $R\subset \ss_8^+$ exists and one can separate in such a  fine way \emph{two distinct Brill-Noether conditions} is due to the existence of a second birational model of $\mm_8$ as a GIT quotient of a certain Grassmannian.

 \vskip 3pt
\ \ \ \ Let $V := \C^6$ and consider the Grassmannian in the Pl\"ucker embedding
$$G := G(2,V) \hookrightarrow \PP(\Lambda^2 V) = \PP^{14}.$$ Then $K_G = \Oo_G(-6)$ and a  general $7$-plane $\PP^7\subset \PP^{14}$ intersects $G$ along a smooth canonical curve of genus 8 with general moduli, see \cite{M2}.

\ \ \ \ Let us fix a point $[C,\eta] \in \Theta_{\text{null}}$. The canonical model $C\subseteq \PP^7$ lies on a rank three quadric $Q_C\in H^0(\PP^7, \mathcal{I}_{C/\PP^7}(2))$.  The quadric $Q_C$ lifts to a quadric $Q_G\in H^0(\PP^{14}, \mathcal{I}_{G/\PP^{14}}(2))$ containing the Grassmannian in its Pl\"ucker embedding.
There is a $6$-dimensional space of extensions of $C$ by a $K3$ surface
$$\begin{array}{ccccc}
  C & \subset & S & \subset & G \\
  \cap &  & \cap &  & \cap \\
  \PP^7 & \subset & \PP^8 & \subset & \PP^{14} \\
\end{array}$$
and for each such extension, the quadric $Q_C$ lifts to a quadric $Q_S \in H^0(\PP^8, \mathcal{I}_{S/\PP^8}(2))$. Note that  $\rank(Q_S) \leq \rank(Q_C) + 2 = 5$, and for a general $K3$ extension  $S\supseteq C$, the equality $\rank(Q_S)=5$ holds.

\begin{proposition}
There is a pencil of K3 extensions of $C\subseteq S \subseteq G$ such that $\rank(Q_S) = 4$.
\end{proposition}
This result is proved in \cite{FV2} and it plays a crucial role in the proof that $\kappa(\ss_8^+)=0$.
One has the following commutative diagram, showing that such a $K3$ surface $S$ is a $7:1$ cover of a smooth quadric $Q_0\subset \PP^3$:

$$\xymatrix{
  S \ar[r]^-f & Q_0 \cong\PP^1\times \PP^1 \ar[dl]_{\pi_1} \ar[dr]^{\pi_2} \subseteq \PP^3 \\
  \PP^1 & & \PP^1    \\
}$$

The $K3$ surface $S$ carries two elliptic pencils $|E_1|$ and $|E_2|$ corresponding to the projections $\pi_1$ and $\pi_2$ and such that $E_i^2 = 0$ for $i=1, 2$. Moreover,  $C\equiv E_1+E_2$ and  $E_1\cdot E_2 = 7$.

\ \ \ \ Let $R$ be the pencil in $\ss_8^+$ obtained by pulling-back via $f$ planes passing through a general line $l_0\subset \PP^3$. Then following \cite{FV2} we write that
$$R\cdot \lambda = \pi_*(R)\cdot \lambda = g+1 = 9$$ and
$$R\cdot(\alpha_0+2\beta_0) = R\cdot\pi^*(\delta_0) =\pi_*(R)\cdot \delta_0= 6(g+3) = 66.$$

There are two reducible fibers in the pencil $R$
corresponding to the planes through $l_0$ spanned by the pairs of rulings of $Q_0$ passing through the points of intersection of $l_0\cap Q_0$. Each of them is counted with multiplicity $\frac{7}{2}$, hence
$$R\cdot \beta_0 = \frac{7}{2} + \frac{7}{2} =7.$$
Therefore we find that  $R\cdot \alpha_0 = 52$, hence
$$R\cdot \ovl{\Theta}_{\text{null}} = \frac{1}{4}R\cdot \lambda - \frac{1}{16}R\alpha_0 = \frac{9}{4} - \frac{52}{16} = -1 < 0$$
and $$R\cdot \pi^*(\ovl{\M}_{8,7}^2) = 0.$$
This completes the proof of the fact that $\kappa(\ss_8^+)=0$.


\begin{thebibliography}{99}



\bibitem[ACV]{ACV} D. Abramovich, A. Corti and A. Vistoli, {\em{Twisted bundles and admissible covers}},
Communications in  Algebra, \textbf{31} (2003), 3547--3618 (special issue in honor of S. Kleiman).

\bibitem[AJ]{AJ} D. Abramovich and T. J. Jarvis,
{\em{Moduli of twisted spin curves}}, Proceedings of the American Math. Society \textbf{131} (2003), 685--699.

\bibitem[ACGH]{ACGH} E. Arbarello, M. Cornalba, P. A. Griffiths, J. Harris
{\em{Geometry of algebraic curves}}, Grundlehren der mathematischen Wissenschaften 267  (1985), Springer Verlag.

\bibitem[An]{An} A. Andreotti, {\em{On a theorem of Torelli}}, American Journal of Mathematics \textbf{80} (1958), 801-828.
\bibitem[AM]{AM} A. Andreotti and A. L. Mayer, {\em{On period relations for abelian integrals on algebraic curvres}}, Annali Scuola Normale Superiore di Pisa, \textbf{3} (1967), 189-238.


\bibitem[B1]{B2} A. Beauville, {\em{Prym varieties and the Schottky
problem}}, Inventiones Math. \textbf{41} (1977), 149-96.

\bibitem[B2]{B3} A. Beauville, {\em{Vari\'{e}t\'{e}s de Prym et jacobiennes interm\'{e}diaires}}, Annales Scientifique \'Ecole Normale Sup\'erieure
\textbf{10} (1977), 309-391.
\bibitem[B3]{B4} A. Beauville, {\em{Prym varieties: A survey}}, in: Theta functions-Bowdoin 1987, Proceedings Symposia Pure Applied Math. Vol. 49 (1989), 607-620.
\bibitem[Bi]{Bi} K.-R. Biermann, {\em{Die Mathematik und ihre Dozenten an der Berliner Universit\"at 1810-1933}}, Berlin, Akademie Verlag 1988.
\bibitem[BV]{BV} A. Bruno and A. Verra, {\em{$\M_{15}$ is
rationally connected}}, in: Projective varieties
with unexpected properties,  51-65, Walter de Gruyter (2005).
\bibitem[BL]{BL} C. Birkenhake and H. Lange, {\em{Complex abelian varieties}}, Grundlehren der mathematischen Wissenschaften  302, 2nd Edition 2004, Springer Verlag.
\bibitem[Bo]{Bo} R. B\"olling, {\em{Weierstrass and some members of his circle}}, in: Mathematics in Berlin 71-82, Birkh\"auser 1998.
\bibitem[Ca]{Ca} F. Catanese, {\em{On the rationality of certain moduli spaces
related to curves of genus $4$}}, Springer Lecture Notes in Mathematics \textbf{1008} (1983), 30-50.
\bibitem[CCC]{CCC} L. Caporaso, C. Casagrande and M. Cornalba, {\em{Moduli of roots of line bundles on curves}},
Transactions American Mathematical Society \textbf{359} (2007), 3733--3768.
\bibitem[CG]{CG} C.H. Clemens and P. A. Griffiths, {\em{The intermediate {J}acobian of the cubic threefold}},
Annals of Mathematics \textbf{95} (1972), 281--356.
\bibitem[Co]{C} M. Cornalba, {\em{Moduli of curves and theta-characteristics}},
in: Lectures on Riemann surfaces ({T}rieste, 1987),
560--589. World Sci. Publ., Teaneck, NJ, 1989.
\bibitem[Cl]{Cl} H. Clemens, {\em{Double solids}}, Advances in Mathematics \textbf{47}
(1983), 107-230.
\bibitem [DP]{DP} C. De Concini and P. Pragacz, {\em{On the class of Brill-Noether loci for Prym varieties}},
Mathematische Annalen \textbf{302} (1995), 687--697.
\bibitem[De]{De} O. Debarre, {\em{Vari\'et\'es de Prym et ensembles d'Andreotti et Mayer}}, Duke Math. Journal
\textbf{60} (1990), 599-630.
\bibitem[Do1]{Do2} I. Dolgachev, {\em{Rationality of $\R_2$ and $\R_3$}}, Pure and Applied Math. Quarterly, \textbf{4} (2008), 501-508.
\bibitem[Do2]{Do3} I. Dolgachev, {\em{Mirror symmetry for lattice polarized $K3$ surfaces}}, Journal Math. Sciences \textbf{81} (1996), 2599-2630.
\bibitem[D1]{D1} R. Donagi, {\em{The unirationality of $\A_5$}},
Annals of Mathematics \textbf{119} (1984), 269-307.
\bibitem[D2]{D2} R. Donagi, {\em{The fibers of the Prym map}},
Contemporary Mathematics \textbf{136} (1992), 55-125.
\bibitem[DS]{DS} R. Donagi and R. Smith, {\em{The structure of the
Prym map}}, Acta Mathematica \textbf{146} (1981), 25-102.
\bibitem[FMP]{FMP} G. Farkas, M. Musta\c{t}\u{a} and M. Popa, {\em{Divisors on
$\M_{g, g+1}$ and the Minimal Resolution Conjecture for points on
canonical curves}}, Annales Scientifique \'Ecole  Normale
Sup\'erieure \textbf{36} (2003), 553-581.
\bibitem[F1]{F1} G. Farkas, {\em{Koszul divisors on moduli spaces of
curves}}, American Journal of Mathematics \textbf{131} (2009),
819-869.
\bibitem[F2]{F2} G. Farkas, {\em{Aspects of the birational geometry of $\mm_g$}}, Surveys in differential geometry, Vol. 14, 2010, 57-110.
\bibitem[F3]{F3} G. Farkas, {\em{The birational type of the moduli space of even spin curves}}, Advances in Mathematics 223 (2010), 433-443.
\bibitem[FL]{FL} G. Farkas and K. Ludwig, {\em{The Kodaira dimension of the moduli
space of Prym varieties}}, Journal of the European Mathematical Society \textbf{12} (2010), 755-795.
\bibitem[FP]{FP} G. Farkas and M. Popa, {\em{Effective divisors on
$\mm_g$, curves on $K3$ surfaces and the Slope Conjecture}}, Journal
of Algebraic Geometry \textbf{14} (2005), 151-174.
\bibitem[FV1]{FV} G. Farkas and A. Verra, {\em{The geometry of the moduli space of odd spin curves}}, arXiv:1004.0278.
\bibitem[FV2]{FV2} G. Farkas and A. Verra, {\em{Moduli of theta-characteristics via Nikulin  surfaces}}, arXiv:1104.0273
\bibitem[FV3]{FV3} G. Farkas and A. Verra, {\em{$\R_8$ is uniruled}}, in preparation.
\bibitem[HF]{HF} H. Farkas, {\em{On the Schottky relation and its generalization to arbitrary genus}}, Annals of Mathematics \textbf{92} (1970), 57-81.
\bibitem[HFR]{FR} H. Farkas and H. Rauch, {\em{Period relations of Schottky type on Riemann surfaces}}, Annals of Mathematics \textbf{92} (1970), 434-461.
\bibitem[Fay]{Fay} J. Fay, {\em{Theta functions on Riemann surfaces}}, Lecture Notes in Mathematics, vol. 352, Springer Verlag 1973.
\bibitem[FS]{FS} R. Friedman and R. Smith, {\em{The generic Torelli
theorem for the Prym map}}, Inventiones Math. \textbf{67} (1982),
437-490.

\bibitem[HM]{HM} J. Harris and D. Mumford, {\em{On the Kodaira dimension of the moduli space of curves}},
Inventiones Math. \textbf{67} (1982), 23--88.
With an appendix by William Fulton.
\bibitem[EH]{EH} D. Eisenbud and J. Harris, {\em{The Kodaira
dimension of the moduli space of curves of genus $\geq 23$}},
Inventiones Math. \textbf{90} (1987), 359-387.
\bibitem[vG]{vG} B. van Geemen, {\em{Siegel modular forms vanishing on the moduli space of curves}}, Inventiones Math. \textbf{78} (1984), 329-349.
\bibitem[vGS]{vGS} B. van Geemen and A. Sarti, {\em{Nikulin involutions on $K3$ surfaces}}, Mathematische Zeitschrift \textbf{255} (2007), 731-753.
\bibitem[GHS]{GHS} V. Gritsenko, K. Hulek and G.K. Sankaran, {\em{The Kodaira dimension of the moduli of $K3$ surfaces}}, Inventiones Math. \textbf{169} (2007), 519-567.
\bibitem[GK]{GK} S. Grushevsky and I. Krichever, {\em{Integrable discrete Schr\"odinger equations and a characterization of Prym varieties by a pair of quadrisecants}}, Duke Math. Journal \textbf{152} (2010), 317-371.
\bibitem[G]{G2} S. Grushevsky, {\em{The Schottky problem}},  arXiv:1009.0369, to appear in: Classical Algebraic Geometry Today, MSRI 2009.
\bibitem[Gr]{Gr} M. Green, {\em{Koszul cohomology and the cohomology of projective varieties}}, J. Differential Geometry \textbf{19} (1984), 125-171.
\bibitem[GL1]{GL1} M. Green and R. Lazarsfeld, {\em{The non-vanishing of certain Koszul cohomology groups}}, J. Differential Geometry \textbf{19}
(1984), 168-170.
\bibitem[GL2]{GL} M. Green and R. Lazarsfeld, {\em{Some results on
the syzygies of finite sets and algebraic curves}}, Compositio
Mathematica \textbf{67} (1988), 301-314.
\bibitem[H]{H} N. Hitchin, {\em{Stable bundles and integrable systems}}, Duke Math. Journal \textbf{54} (1987), 91-114.
\bibitem[IGS]{IGS} E. Izadi, M. Lo Giudice and G. Sankaran, {\em{The
moduli space of \'etale double covers of genus $5$ is unirational}},
Pacific Journal of Mathematics \textbf{239} (2009), 39-52.
\bibitem[IL]{IL} E. Izadi and H. Lange, {\em{Counterexamples of high Clifford-index to Prym-Torelli}},  arXiv:1001.3610.

\bibitem[J]{J} T. J. Jarvis, {\em{Geometry of the moduli of higher spin curves}},
International Journal of Mathematics \textbf{11} (2000), 637--663.

\bibitem[Ka]{Ka} P. Katsylo, {\em{On the unramified $2$-covers of curves of genus $3$}}, in: Algebraic Geometry and Applications
(Yaroslavl 1992), Aspects of Mathematics Vol. E25 (1994), 61-65.
\bibitem[Kl]{Kl} F. Klein, {\em{\"Uber Riemanns Theorie der algebraischen Funktionen und ihrer Integrale}}, Gesammelte mathematische Abhandlungen 3, 499-573.
\bibitem[Kr]{Kr} I. Krichever, {\em{Characterizing Jacobians via  trisecants of the Kummer variety}}, Annals of Mathematics \textbf{172} (2010), 485-516.
\bibitem[Kra]{Kra} A. Krazer, {\em{Friedrich Prym}}, Jahresbericht der Deutschen Mathematiker-Vereinigung \textbf{25} (1917), 1-15.
\bibitem[La]{La} D. Laugwitz, {\em{Bernhard Riemann 1826-1866, Turning Points in the Conception of Mathematics}}, translated by Abe Shenitzer, Birkh\"auser 1999.
\bibitem[Lo]{Lo} E. Looijenga, {\em{Smooth Deligne-Mumford compactifications by means  of Prym level structures}}, Journal of Algebraic Geometry
\textbf{3} (1994), 283-293.
\bibitem[MM]{MM} S. Mori and S. Mukai, {\em{The uniruledness of the
moduli space of curves of genus $11$}}, Springer Lecture Notes in
Mathematics  \textbf{1016} (1983),  334–353.
\bibitem[M1]{M1} S. Mukai, {\em{Curves, $K3$ surfaces and Fano $3$-folds of genus $\leq 10$}}, in: Algebraic geometry and commutative algebra
(1988), 357-377, Kinokuniya, Tokyo.
\bibitem[M2]{M2} S. Mukai, {\em{Curves and Grassmannians}}, in: Algebraic Geometry and Related Topics, Inchon 1992,  (J.-H. Yang, Y. Namikawa, K. Ueno, editors),  1992, 19-40, International Press.
\bibitem[M]{M} D. Mumford, {\em{Prym varieties I}}
in: Contributions to analysis (L. V. Alfors et al. eds.),
Academic Press, New York (1974), 325-350.
\bibitem[N]{N} B.C. Ng\^o, {\em{Le lemme fondamental pour les alg\'{e}bres de Lie}}, Publ. Mathematiques Inst. Hautes \'Etudes Scientifique \textbf{111} (2010), 1-169.
\bibitem[Ni]{Ni} V.V. Nikulin, {\em{Kummer surfaces}}, Izvestia Akad. Nauk SSSR \textbf{39} (1975), 278-293.
\bibitem[P1]{P1} F. Prym, {\em{Zur Theorie der Funktionen in einer zweibl\"attrigen Fl\"ache}}, Denkschrift der Schweiz. Naturforschenden  Gesellschaft, Bd. 22, 1866.
\bibitem[P2]{P2} F. Prym, {\em{Zur Integration der gleichzeitigen Differentialgleichungen $\frac{\partial^2 u}{\partial x^2}+\frac{\partial^2 u}{\partial y^2}=0$}}, J. reine angewandte Mathematik \textbf{73} (1871), 340-364.
\bibitem[PR]{PR} F. Prym and G. Rost, {\em{Theorie der Prymschen Funktionen erster Ordnung im Anschluss an die Sch\"opfungen Riemanns}}, 1911.
\bibitem[R]{R} M. Raynaud, {\em{Sections des fibr\'es vectoriels sur
une courbe}}, Bulletin Soc. Math. France \textbf{110} (1982), 103-125.
\bibitem[Re]{Re}
M. Reid, {\em{Canonical {$3$}-folds}}, in Journ\'{e}es de
G\'{e}ometrie Alg\'{e}brique d'Angers,
  Juillet 1979/Algebraic Geometry, Angers 1979, Sijthoff \&
  Noordhoff, Alphen aan den Rijn (1980), 273-310.
\bibitem[Re2]{re2002}
M. Reid, {\em{La correspondance de {M}c{K}ay}}, S\'eminaire Bourbaki,
Vol.\ 1999/2000, Ast\'erisque, \textbf{276} (2002), 53-72.
\bibitem[Sch]{Sch} F. Schottky, {\em{Zur Theorie der Abelschen Funktionen von vier Variabeln}}, J. reine angewandte Mathematik, \textbf{102} (1888), 304-352.
\bibitem[SJ]{SJ} F. Schottky and H. Jung, {\em{Neue S\"atze \"uber Symmetralfunktionen und die Abelschen
Funktionen der Riemannschen Theorie}}, S.-B.  Preuss. Akad. Wiss. Berlin, Phys. Math. Kl. 1 (1909), 282-287.
\bibitem[T]{ta1982}
Y. Tai, {\em{On the Kodaira dimension of the moduli space of abelian
varieties}}, Inventiones Math. \textbf{68} (1982), 425-439.
\bibitem[V1]{V1} A. Verra, {\em{A short proof of the unirationality
of $\A_5$}}, Indagationes Math. \textbf{46} (1984), 339-355.
\bibitem[V2]{V2} A. Verra, {\em{On the universal principally polarized abelian variety of dimension $4$}}, in: Curves and abelian varieties (Athens, Georgia, 2007)  Contemporary Mathematics \textbf{465} (2008), 253-274.
\bibitem[Vo]{Vo} C. Voisin, {\emph{Green's generic syzygy conjecture
for curves of even genus lying on a $K3$ surface}}, Journal of
European Mathematical Society \textbf{4} (2002), 363-404.
\bibitem[Vol]{Vol} H.-J. Vollrath, {\em{Friedrich Prym (1841-1915)}}, in: Lebensbilder bedeutender W\"urzburger Professoren (P. Baumgart editor), Neustadt/Aisch 1995, 158-177.
\bibitem[Wi]{Wi} W. Wirtinger, {\em{Untersuchungen \"uber Thetafunktionen}}, Teubner, Berlin 1895.
\bibitem[We]{W} G. Welters, {\em{A theorem of Gieseker-Petri type for Prym varieties}},
Annales Scientifique \'Ecole Normale Sup\'erieure  \textbf{18} (1985), 671--683.


\end{thebibliography}
\end{document}